\documentclass[]{amsart}
\usepackage{latexsym,amssymb,amsmath,amsthm}
\usepackage{comment}
\usepackage{float}
\usepackage{graphicx}
\newtheorem{thm}{Theorem}
\newtheorem{coro}[thm]{Corollary}
\newtheorem{lemma}[thm]{Lemma}
\newtheorem{propo}[thm]{Proposition}
\theoremstyle{definition}

\newcommand{\Z}{\mathbb{Z}}

\def\G{\mathcal{G}}

\def\Det{\rm{Det}}

\def\tr{\rm{tr}}

\title{The zeta function for circular graphs}
\author{Oliver Knill}
\date{Dec 15, 2013}
\address{
        Department of Mathematics \\
        Harvard University \\
        Cambridge, MA, 02138
        }

\begin{document}
\maketitle
\begin{abstract}
We look at entire functions given as the zeta function $\sum_{\lambda>0} \lambda^{-s}$, where $\lambda$
are the positive eigenvalues of the Laplacian of the discrete circular graph $C_n$.
We prove that the roots converge for $n \to \infty$ to the line $\sigma={\rm Re}(s)=1/2$ in the sense that
for every compact subset $K$ in the complement of that line, there is a $n_K$ such that for $n>n_K$, no
root of the zeta function is in $K$. To prove the result we actually look at the Dirac zeta function
$\zeta_n(s) = \sum_{\lambda>0} \lambda^{-s}$ where $\lambda$ are the positive eigenvalues of the Dirac 
operator of the circular graph. In the case of circular graphs, the  Laplace zeta function is 
$\zeta_n(2s)$. 
\end{abstract}

\section{Extended summary}

The zeta function $\zeta_G(s)$ for a finite simple graph $G$ is the entire function
$\sum_{\lambda>0} \lambda^{-s}$ defined by the positive eigenvalues $\lambda$ of the
Dirac operator $D=d+d^*$, the square root of the Hodge Laplacian $L=d d^* + d^* d$ on discrete differential forms.
We study it in the case of circular graphs $G=C_n$, where $\zeta_n(2s)$ agrees with
the zeta function 
$$  \sum_{\lambda >0, \lambda \in \sigma(L)} \lambda^{-s} $$
of the classical scalar Laplacian like
$$ L =  \left[
                 \begin{array}{ccccccccc}
                  2 & -1 & 0 & 0 & 0 & 0 & 0 & 0 & -1 \\
                  -1 & 2 & -1 & 0 & 0 & 0 & 0 & 0 & 0 \\
                  0 & -1 & 2 & -1 & 0 & 0 & 0 & 0 & 0 \\
                  0 & 0 & -1 & 2 & -1 & 0 & 0 & 0 & 0 \\
                  0 & 0 & 0 & -1 & 2 & -1 & 0 & 0 & 0 \\
                  0 & 0 & 0 & 0 & -1 & 2 & -1 & 0 & 0 \\
                  0 & 0 & 0 & 0 & 0 & -1 & 2 & -1 & 0 \\
                  0 & 0 & 0 & 0 & 0 & 0 & -1 & 2 & -1 \\
                  -1 & 0 & 0 & 0 & 0 & 0 & 0 & -1 & 2 \\
                 \end{array}
                 \right] $$
of the circular graph $C_9$. To prove that all roots of the zeta function of $L$ converge for $n \to \infty$ to 
the line $\sigma=1/2$, we show that all roots of $\zeta_n=\zeta_{C_n}$ accumulate for $n \to \infty$ 
on the line $\sigma={\rm Re}(s) = 1$.
The strategy is to deal with three regions: we first verify that for ${\rm Re}(s)>1$ the analytic 
functions $\zeta_n(s) \pi^s/(2 n^s)$ converge to the classical Riemann zeta function which is nonzero 
there. Then, we show that $(\zeta_n(s) - n c(s))/n^s$ converges uniformly on compact sets for $s$ in 
the strip $0<{\rm Re}(s) < 1$, where $c(s)=\int_0^1 \sin(\pi x)^{-s} \; dx$. 
Finally, on ${\rm Re}(s)\leq 0$ we use that $\zeta_n(s)/n$ converges to the average $c(s)$ as a Riemann sum. 
While the later is obvious, the first two statements need some analysis. 
Our main tool is an elementary but apparently new Newton-Coates-Rolle analysis which considers
numerical integration using a new derivative $K_f(x)$ called {\bf K-derivative} which has very similar
features to the Schwarzian derivative. The derivative
$K_f(x)$ has the property that $K_f(x)$ is bounded for $f(x)=\sin(\pi x)^{-s}$ if $s \neq 1$. \\

We derive some values $\zeta_n(2k)$ explicitly as rational numbers 
using the fact that $f(x)=\cot(\pi x)$ is a fixed point of the linear Birkhoff renormalization map 
\begin{equation}
Tf(x)= \frac{1}{n} \sum_{k=0}^{n-1} f(\frac{x}{n} + \frac{k}{n}) \; . 
\label{renormalization}
\end{equation}
This gives immediately the discrete zeta values like the {\bf discrete Basel problem}
$\zeta_n(2)=\sum_{k=1}^{n-1} \lambda_k^{-2} = (n^2-1)/12$ 
which recovers in the limit the classical Riemann zeta values like the 
{\bf classical Basel problem} $\zeta(2)=\sum n^{-2}=\pi^2/6$. 
The exact value $\zeta_n(2)$ in the discrete case is of interest because it is the trace
${\rm Tr}(L^{-1})$ of the Green function of the Jacobi matrix $L$, the Laplacian of the circular 
graph $C_{n}$. The case $s=2$ is interesting especially because 
$-\pi^2 \sin^{-2}(\pi x) = \zeta'(1,x) - \zeta'(1,1-x)$, where $\zeta(s,x)$ is
the Hurwitz zeta function. 
And this makes also the relation with the $\cot$ function clear, 
because $\pi \cot(\pi x) = \zeta(1,x) + \zeta(1,1-x)$ by the cot-formula of Euler.  \\

The paper also includes high precision numerical computations of roots of $\zeta_n$ made to investigate 
at first where the accumulation points are in the limit $n \to \infty$. These computations show that 
the roots of $\zeta_n$ converge very slowly to the critical line $\sigma=1$ and gave us confidence to 
attempt to prove the main theorem stated here. \\

We have to stress that while the entire functions $\zeta_n(s)  \pi^s/(2 n^s)$ approximate the 
Riemann zeta function $\zeta(s)$ for $\sigma={\rm Re}(s)>1$, {\bf there is no direct relation} on and 
below the critical line ${\rm Re}(s)=1$. The functions $\zeta_n$ continue to be analytic everywhere, 
while the classical zeta function $\zeta$ continues only to exist through analytic continuation. 
So, the result proven here {\bf tells absolutely nothing about the Riemann hypothesis}. 
While for the zeta function of the circle, everything behind the abscissa of convergence is foggy 
and only visible indirectly by analytic continuation, we can fly in the discrete case with full 
visibility because we deal with entire functions.  \\

There are conceptual relations however between the classical and discrete zeta functions:
$\zeta_n$ is the Dirac zeta function of the discrete circle $C_n$ and $\zeta$ is the Dirac zeta function of 
the circle $T^1=R^1/Z^1$ which is the Riemann zeta function. 
At first, one would expect that the roots of $\zeta_n$ approach $\sigma=1/2$; but this
is not the case: they approach $\sigma=1$. Only the Laplace zeta function $\zeta_n(2s)$ has the 
roots near $\sigma=1/2$. 
%the discrete analogue of the Hurwitz zeta function $\zeta_n(s,x) = \sum_{k=0}^{n-1} \sin((k+x)/n)^{-s}$ 
%has the property that when normalized to have with zero mean and $L^2$ norm $1$ converges for 
%$n \to \infty$ to the normalization of $\zeta(s,x) + \zeta(s,1-x)$.
In the proof, we use that the Hurwitz zeta functions $\zeta(s,x)$ 
are fixed points of Birkhoff renormalization operators $Tf(x) = (1/n^s) \sum_{k=0}^{n-1} f((k+x)/n)$, 
a fact which implies the already mentioned fact that $\cot(\pi x)$ fixed point property as a special case
because $\lim_{s \to 1} (\zeta(s,x)-\zeta(s,1-x)) = \pi \cot(\pi x)$ by Euler's cotangent formula. 
The fixed point property of the Hurwitz zeta function is obvious for ${\rm Re}(s)>1$ and clear
for all $s$ by analytic continuation. It is known as a summation formula. \\

Actually, we can look at Hurwitz zeta functions
in the context of a central limit theorem because Birkhoff summation can be seen as a normalized sum of 
random variables even so the random variables $X_k(x) = f((x+k)/n)$ in the sum $S_n = X_1+ \dots + X_n$ 
depend on $n$. For smooth functions which are continuous up to the boundary, the normalized sum converges
to a linear function, which is just $\zeta(0,x)$ normalized to have standard deviation $1$. 
When considered in such a probabilistic setup, Hurwitz Zeta functions play a similar "central"
role as the Gaussian distribution on $R$, the exponential distribution on $R^+$ or Binomial distributions 
on finite sets. 

\section{Introduction}

The symmetric zeta function 
$\sum_{\lambda \neq 0} \lambda^{-s}$ of a finite simple graph $G$ is defined by the nonzero eigenvalues $\lambda$ 
of the Dirac matrix $D=d+d^*$ of $G$, where $d$ is the exterior derivative. 
Since the nonzero eigenvalues come in pairs $-\lambda,\lambda$, all the important information like 
the roots of $\zeta(s)$ are encoded by the entire function 
$$  \zeta_G(s) = \sum_{\lambda > 0} \lambda^{-s} \; , $$
which we call the {\bf Dirac zeta function} of $G$. It is the discrete analogue of the Dirac zeta 
function for manifolds, which for the circle $M=T^1$, where $\lambda_n=n, n \in \Z$, agrees with the {\bf classical 
Riemann zeta function} $\zeta(s) = \sum_{n=1}^{\infty} n^{-s}$. The analytic function $\zeta_G(s)$ has infinitely 
many roots in general, which appear in computer experiments arranged in a strip of the complex plane.
Zeta functions of finite graphs are entire functions
which are not trivial to study partly due to the fact that for fixed $\sigma={\rm Re}(s)$, the function $\zeta(\sigma + i \tau)$ 
is a Bohr almost periodic function in $\tau$, the frequencies being related to the Dirac eigenvalues of the graph. \\

There are many reasons to look at zeta functions, both in number theory as well as in physics 
\cite{kirsten,Elizalde}. One motivation can be that in statistical mechanics, one looks at partition functions
$Z(t)=\sum \exp(-t \lambda)$ of a system with energies $\lambda$. 
Now, $Z'(0) = -\sum (\lambda_k)$ and the $s$'th derivative in a distributional sense is 
$d^s/dt^s Z(0) = \sum_{\lambda} \lambda^s$. If $\lambda$ are the positive eigenvalues of a matrix $A$,
then $Z(t) = {\rm tr}( e^{-t A})$ is a heat kernel. An other motivation is that $\exp(-\zeta'(0))$ is the pseudo 
determinant of $A$ and that the analytic function, like the characteristic polynomial, encodes 
the eigenvalues in a way which allows to recover geometric properties of the graph. As we will see, there 
are interesting complex analytic questions involved when looking at {\bf graph limits}
of {\bf discrete circles} $C_n$ which have the circle as a limit.  \\

This article starts to study the zeta function for such circular graphs $C_n$. The function
\begin{equation}
 \zeta_n(s) = \sum_{k=1}^{n-1} 2^{-s} \sin^{-s}(\pi \frac{k}{n}) 
\end{equation}
is of the form $\sum_{k=1}^{n-1} g(\pi k/n)$ with the complex function $g_s=(2\sin(x))^{-s}$ and
where $\Det(D(C_n))=n$. If $s$ is in the strip $0<\sigma<1$, the function $x \to g_s(x)$ is in $L^1$, but not bounded. 
If $c(s)=\int_0^1 g_s(x) \; dx$, then $\zeta_n(s) - n c(s)$ is a Riemann sum $\sum_{k=1}^{n-1} h(\pi k/n)$ for
some function $h(x)$ on the circle $R/Z$ which satisfies $\int_0^1 h(x) \; dx =0$.
Whenever $h$ is continuous like for $\sigma={\rm Re}(s) <0$, the Riemann
sum $(1/n) \sum_{k=1}^{n-1} h(\pi k/n)$ converges. However, if $h$ is unbounded, this is not
necessarily true any more because some of the points get close to the singularity of $g$. But if - as in our case - 
the Riemann sum is evaluated on a nice grid which only gets $1/n$ close
to a singularity of type $x^{-s}$, one can say more. 
In the critical strip the function is in $L^1(T^1)$. To the right of the critical line $\sigma=1$, the function 
$g_s$ has the property that $\zeta_n(s) \pi^s/(2 n^s)$ converges to the classical Riemann zeta function $\zeta(s)$.
Number theorists have looked at more general sums. Examples are \cite{HL,HL46}. \\

Studying the zeta functions $\zeta_n$ for discrete circles $C_n$ allows us
to compute values of the classical zeta function $\zeta(s)$ in a new way. To do so, 
we use a result from \cite{knillcotangent} telling that $\cot$ is a fixed point of a Birkhoff renormalization 
map (\ref{renormalization}).
We compute $\zeta_{m}(s)$ explicitly for small $k$, solving the Basel problem for circular 
graphs and as a limit for the circle, where it is the classical Basel problem. While these sums have been summed up
before \cite{Dowker}, the approach taken here seems new. \\

To establish the limit $[\zeta_n(s)-n c(s)]/n^s$, we use an adaptation of a Newton-Coates method for 
$L^1$-functions which are smooth in the interior of an interval but unbounded at the boundary. 
This method works, if the {\bf K-derivative}
$$ K_f = \frac{f' f'''}{(f'')^2} $$ 
is bounded $|K_f(x)| \leq M$ for $x \in (0,1)$. The K-derivative has properties similar to the {\bf Schwarzian derivative}.
It is bounded for functions like $\cot(\pi x), \log(x),x^{-s}$ or $(\sin(\pi x))^{-s}$ which appear in the zeta 
function of the discrete circle $C_n$. A comparison table below shows the striking similarities with the Schwarzian derivative.
On the critical line $\sigma=1$, we still see that $\zeta_n(s)/(n \log(n))$ stays bounded. \\

\begin{figure}
\scalebox{0.21}{\includegraphics{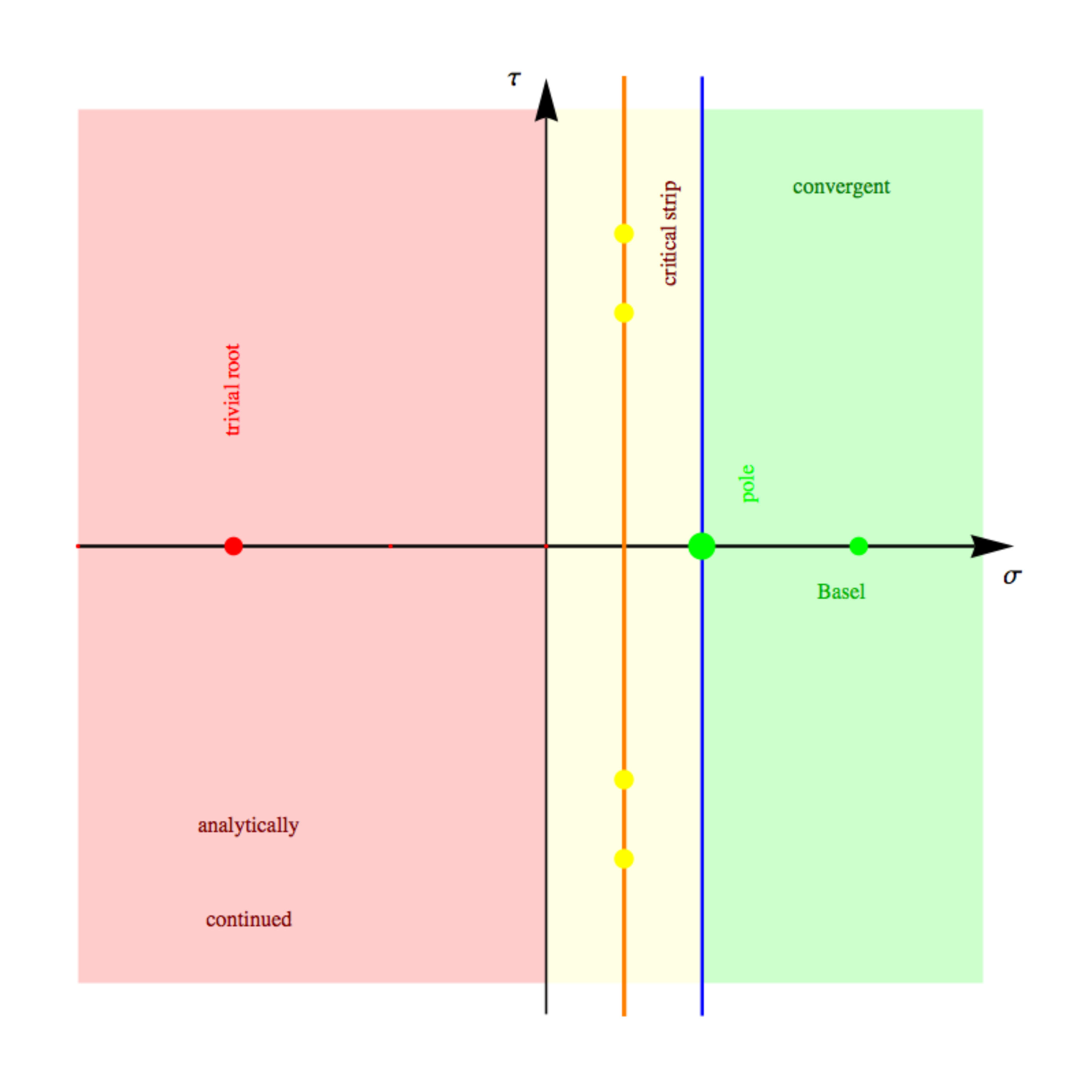}}
\scalebox{0.21}{\includegraphics{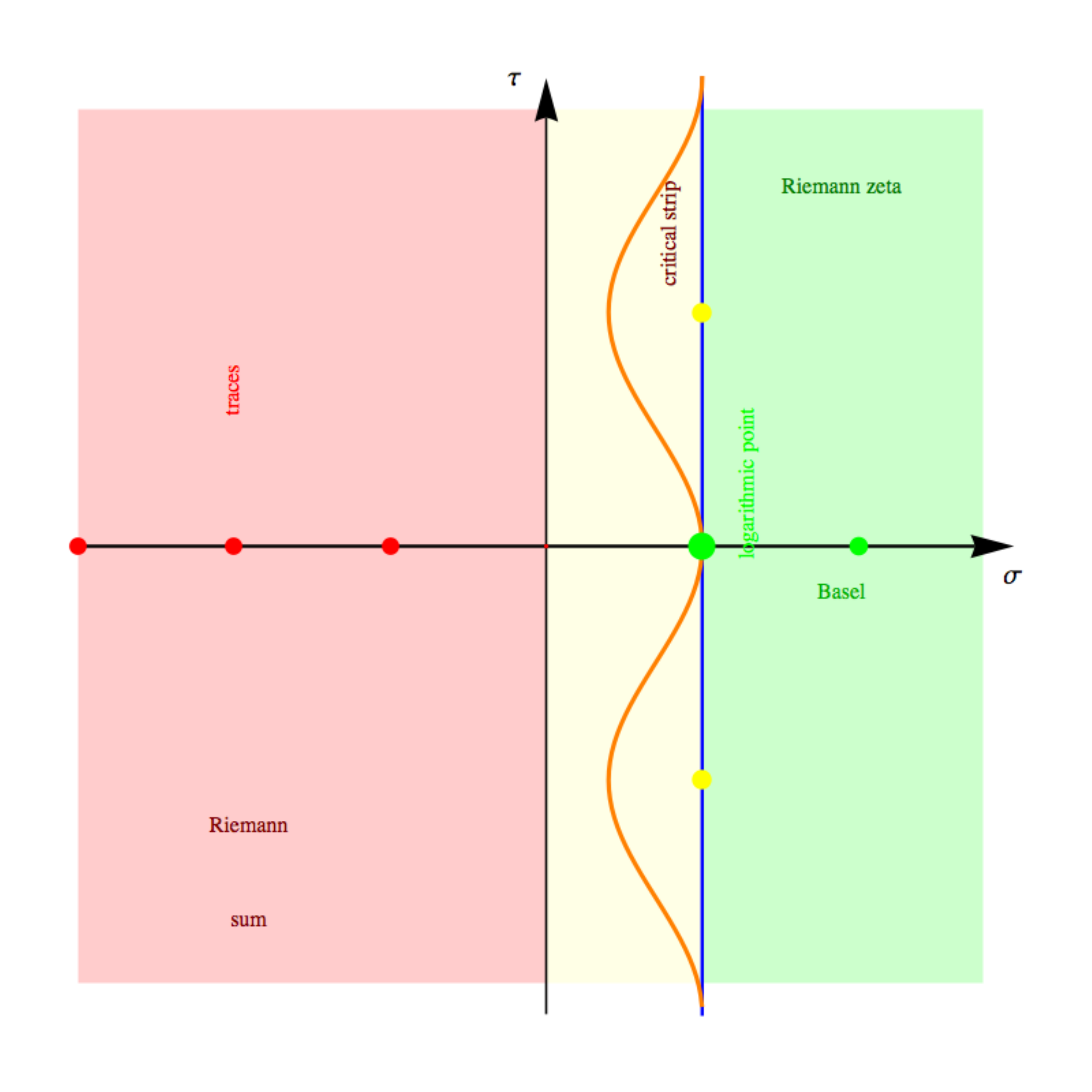}}
\caption{
The Zeta function $\zeta(s)$ of the circle $T^1$ is the classical Riemann zeta function
$\zeta(s) = \sum_{n=1}^{\infty} n^{-s}$. The Zeta function $\zeta_n(s)$ of the circular graph $C_n$ is the discrete
circle zeta function $\zeta_n(s) = \sum_{k=1}^{n-1} (2 \sin(\pi k/n))^{-s}$ which is a finite sum. 
In the limit $n \to \infty$, the roots of the entire functions $\zeta_n$ are on the curve $\sigma=1$.
}
\end{figure}

Various zeta functions have been considered for graphs: there is the discrete analogue of the 
{\bf Minakshisundaraman-Pleijel zeta function}, which uses the eigenvalues of the scalar Laplacian. In a discrete
setting, this zeta function has been mentioned in \cite{fanchunglattice} for the scalar San Diego Laplacian. 
Related is $\sum_{\lambda>0} \lambda^{-s}$,
where $\lambda$ are the eigenvalues of a graph Laplacian like the {\bf scalar Laplacian} $L_0 = B-A$, 
where $B$ is the diagonal 
{\bf degree matrix} and where $A$ is the {\bf adjacency matrix} of the graph. 
In the case of graphs $C_n$, the Dirac zeta function is not different from the Laplace zeta function due to 
Poincar\'e duality which gives a symmetry between 0 and 1-forms so that the Laplace and Dirac zeta function
are related by scaling. If the full Laplacian on forms is considered for $C_n$, this can be written as 
$\zeta(2s)$ if $\zeta(s)$ is the Dirac zeta function for the manifold. 
The {\bf Ihara zeta function} for graphs is the discrete analogue of the {\bf Selberg
zeta function} \cite{Ihara,Cooper,Terras} and is also pretty unrelated to the Dirac zeta function considered here.
An other unrelated zeta function is the {\bf graph automorphism zeta function} \cite{brouwergraph} for graph 
automorphisms is a version of the {\bf Artin-Mazur-Ruelle zeta function} in dynamical systems theory \cite{ruellezeta}.
We have looked at {\bf almost periodic Zeta functions} $\sum_n g(n \alpha)/n^s$ with periodic function $g$ and 
irrational $\alpha$ in \cite{KL}. An example is the polylogarithm $\sum_n \sin(\pi n \alpha)/n^s$. 
Since we look here at $\sum_n \sin^{-s}(n \alpha)$ with rational $\alpha$, one could look at 
$\sum_{k=1}^q \sin^{-s}(\pi k \alpha)$, where $p/q$ are periodic approximations of $\alpha$. For 
special irrational numbers like the golden mean, there are symmetries and the limiting values can in some
sense computed explicitly \cite{KL,KT,knillcotangent}. Fixed point equations of Riemann sums are a common ground
for Hurwitz zeta functions, polylogarithms and the $\cot$ function. \\

In this article we also report on experiments which have led to the main question left open:
to find the limiting set the roots converge to in the critical strip. We will prove that the limiting set is a
subset of the critical line $\sigma=1$ but we have no idea about the nature of this set, whether it is 
discrete or a continuum. Figure~(\ref{situation}) shows the motion of the roots from $n=100$ to $n=10000$ with linearly
increasing $n$ and then from $10000$ to $20480000$, each time doubling $n$. There is a reason for the 
doubling choice: $\zeta_{2n}$ can be written as a sum $\zeta_n$ and a twisted $\zeta_n(1/2)$
and a small term disappearing in the limit. This allows to interpolate roots for $n$ and $2n$ with a continuous
parameter seeing the motion of the roots in dependence of $n$ as a flow on time dependent vector field. This
picture explains why the roots move so slowly when $n$ is increased. 

\section{The Dirac zeta function} 

For a finite simple graph $G=(V,E)$, the Dirac zeta function $\zeta_G(s) = \sum_{\lambda>0} \lambda^{-s}$
is an analytic function. It encodes the spectrum  $\{ \lambda \; \}$ 
of $|D|$ because the traces of $|D|^n$ are recoverable as $\tr(|D|^n)/2 = \zeta(-n)$. Therefore, because
the spectrum of $D$ is symmetric with respect to $0$, we can recover the 
spectrum of $D$ from $\zeta_G(s)$. 
The Hamburger moment problem assures that the eigenvalues are determined from those traces.
For circular graphs $C_n$, the eigenvalues are 
$$   \lambda_k =  \pm 2 \sin(\pi k/n),  \;  k=1 \dots n-1 $$ 
together with two zero 
eigenvalues which belong by Hodge theory to the Betti numbers $b_0=b_1=1$ of $C_n$ for $n \geq 4$. 
The Zeta function $\zeta(s)$ is a Riemann sum for the integral 
$$  c(s) = \int_0^1 2^{-s} \sin^{-s}(\pi x) \; dx =  2^{-s} \frac{\Gamma((1-s)/2)}{\sqrt{\pi}\Gamma(1-s/2)} \; . $$

Besides the circular graph, one could look at other classes of graphs and study the limit. It is not always 
interesting as the following examples show: for the complete graph $K_n$, the nonzero Dirac eigenvalues are $\sqrt{n}$ with 
multiplicity $2^{(2^{n-1}-1)}$ and $\zeta_n(s) = 2^{(2^{n-1}-1)} n^{-s/2}$ do not have any roots.
For star graphs $S_n$, the nonzero Dirac eigenvalues are $1$ with multiplicity $n-1$ and $\sqrt{n}$ 
with multiplicity $1$ so that the zeta function is 
$\zeta_n(s) = (n-1) + n^{-s/2}$ which has a root at $-\frac{2 \log(1-n)}{\log (n)}$. One particular case one could
consider are discrete tori approximating a continuum torus.  Besides graph limits, 
we would like to know the statistics of the roots for random graphs. Some pictures of zeta functions of various
graphs can be seen in Figure~(\ref{diracgraphexamples}). As already mentioned, we stick here to the zeta functions 
of the circular graphs. It already provides a rich variety of problems, many of which are unsettled. \\

Unlike in the continuum, where the zeta function has more information hidden behind the line of convergence - 
the key being analytic continuation -  which allows for example to define notions like Ray-Singer determinant 
as $\exp(-\zeta'(0))$, this seems not that interesting at first in the graph case because the Pseudo determinant 
\cite{cauchybinet} is already explicitly given as the product of the nonzero eigenvalues. Why then study the zeta functions 
at all? One reason is that for some graph limits like graphs approximating Riemannian manifolds, 
an interesting convergence seems to take place and that in the 
circle case, to the right of the critical strip, a limit of $\zeta_n$ leads to the classical Riemann zeta function 
$\zeta$ which is the Dirac zeta function of the circle. As the Riemann hypothesis illustrates, already this 
function is not understood yet. To the left of the critical strip, we have convergence too and the
affine scaling limit $\chi(s) = \lim_{n \to \infty} (\zeta_n(s)-n c(s))/n^s$ exists in the critical strip, a result
which can be seen as a central limit result. Still, the limiting behavior of the roots on the critical line
is not explained. While it is possible that more profound relations between the classical Riemann zeta function and 
the circular zeta function, the relation for $\sigma>1$ pointed out here is only a trivial bridge. It is possible that
the fine structure of the limiting set of roots of circular zeta functions encodes information of the Riemann zeta function
itself and allow to see behind the abscissa of convergence of $\zeta(s)$ but there is no doubt that such an analysis 
would be much more difficult.  \\

\begin{comment}
c[s_] := 2^(-s) Gamma[(1-s)/2]/(Sqrt[Pi]*Gamma[1 - s/2]);
F[n_,s_] := Sum[2^(-s) N[Sin[Pi k/n]^(-s)], {k,1,n - 1}];
h[n_, s_] := (F[n, s] - n c[s])/n^s; H[s_] := h[40000, s];
\end{comment}

Approximating the classical Riemann zeta function by entire functions is possible in many ways.
One has studied partial sums $\sum_{k=1}^n n^{-s}$ in \cite{GonekLedoan,Mora}. 
One can look at zeta functions of triangularizations of compact manifolds $M$ and compare them with the
zeta functions of the later. This has been analyzed for the form Laplacian by \cite{Mantuano} in dimensions
$d \geq 2$.  We will look at zeta functions of circular graphs $C_n(s)$ and their symmetries in form of 
fixed points of Birkhoff sum renormalization maps (\ref{renormalization}),
and get so a new derivation of the values of $\zeta(2n)$ for 
the classical Riemann zeta function. While there is a more general connection for $\sigma>1$ for the
zeta function of any Riemannian manifold $M$ approximated by graphs assured by \cite{Mantuano}, 
we only look at the circle case here, where everything is explicit.  We will prove

\begin{thm}[Convergence]
a) For any $s$ with ${\rm Re}(s)>1$ we have point wise 
$$  \zeta_n(s) \frac{\pi^s}{2 n^s} \to \zeta(s)  \; . $$ 
b) For ${\rm Re}(s) \leq 0$ we have point wise convergence
$$  \frac{\zeta_n(s)}{n} \to c(s) \neq 0  \; . $$
c) For $0 < {\rm Re}(s) < 1$, we have point wise convergence 
of $[\zeta_n(s) -n c(s)]/\sigma(\zeta_n)$ to a nonzero function. \\
In all cases, the convergence is uniform on compact subsets of the 
corresponding open regions. 
\label{maintheorem}
\end{thm}

This will be shown below. An immediate consequence is: 

\begin{coro}[Roots approach critical line]
For any compact set $K$ and $\epsilon>0$, there exists $n_0$ so that for 
$n>n_0$, the roots of $\zeta_n(s)$ are outside 
$K \cap \{ 1-\epsilon < {\rm Re}(s) < 1+\epsilon \; \}$.
\end{coro}

The main question is now whether
$$  \zeta_n(s) = \sum_{k=1}^{n-1} \frac{1}{2^s \sin^{s}(\pi \frac{k}{n})} \; $$
has roots converging to a discrete set, to a Cantor type set  or to a curve $\gamma$ 
on  the critical line ${\rm Re}(s)=1$.  \\

If the roots should converge to a discrete set of isolated points, one could settle this 
with the argument principle and show that a some complex integral $\int_C f'(s)/f(s) \; ds$ 
along some rectangular paths do no more move.  
Besides a curve or discrete set $\{b_1,b_2, \dots \; \}$ the roots could  accumulate as a 
Hausdorff limit on compact set or Cantor set but the later is not likely. 
From the numerical data we would not be surprised if the limiting set would coincide
with the critical line $\sigma=1$. \\

We have numerically computed contour integrals 
$$ \frac{1}{2\pi i} \int_C frac{\zeta'(z)}{\zeta(z)} \; dz  $$
counting the number of roots of $\zeta_n$ inside rectangles
enclosed by  $C: (0,a) \to (1,a) \to (1,b) \to (0,b)$ to confirm the location of roots. That works reliably
but it is more expensive than finding the roots by Newton iteration. 
In principle, one could use methods developed in 
\cite{JohnsonTucker} to rigorously establish the existence of roots for specific 
$n$ but that would not help since we want to understand the limit $n \to \infty$.
If the roots should settle to a discrete set, then \cite{JohnsonTucker} would be useful to prove that this is the case: 
look at a contour integral along a small circle along the limiting root and show that it does not change in the limit 
$n \to \infty$. We currently have the impression however that the roots will settle to a continuum curve on $\sigma=1$. 

\section{Specific values}

We list now specific values of 
$$ \zeta_n(s) = \frac{1}{2^s \sin^s(\pi \frac{1}{n})} + \cdots + \frac{1}{2^s \sin^s(\pi \frac{n-1}{n})} \; . $$
As a general rule, it seems that one can compute the values $\zeta_n(s)$ for the discrete circle explicitly,
if and only the specific values for the classical Riemann zeta function $\zeta(s)$ of the circle are known. 
So far, we have only anecdotal evidence for such a meta rule. 

\begin{lemma}
For real integer values $s=m>0$, we have 
$$  \zeta_n(-2m) = \frac{\tr(L^m)}{2} = n \sum_{k=0}^m B(m,k)^2 \; ,  $$
where $L$ is the form Laplacian of the circular graph $C_n$. 
\end{lemma} 
\begin{proof}
Start with a point $(0,0)$ and take a simplex $x$ (edge or vertex), then choosing a nonzero entry of $Dx$ is 
either point to the right or up. In order to have a diagonal entry of $D^m$, we have to hit $m$ times $-1$ 
and $m$ times $+1$. 
\end{proof}

The traces of the Laplacian have a path interpretation if the diagonal entries are added as loops of negative
length. For the adjacency matrix $A$ one has that ${\rm tr}(A^n)$ is the number of closed paths of length
$n$ in the graph. We can use that $L=D^2$ and $|D|$ is the adjacency matrix of a simplex 
graph $\G$. But it is not $|D|^2=D^2$. 

For $n=0$, where $\tr(D^0)=2n$ but $\zeta_n(0) = 2n-2$, we have a discrepancy;
the reason is that we have defined the sum $\sum_{\lambda \neq 0} \lambda^{-s}$ 
which for $s=0$ does not count two cases of $0^0$, which the trace includes. 
Formulas like $2 \sum_{k=1}^{n-1} 2^m \sin^m(\pi k/n) = {\rm tr}(D^m)$ relate combinatorics
with analysis. It is Fourier theory in the circular case. 

\begin{itemize}
\item For $s=0$ we have $\zeta_n(s) = 2n-2$ because the sum has $2n-2$ terms $1$. 
\item For $s=-1$ we have $\sum_{k=1}^{n-1} 2 \sin(\pi k/n) = {\rm Im}(4/(1-e^{i \pi/n})$
which as a Riemann sum converges for $n \to \infty$ 
to $2 \Gamma(1)\sqrt{\pi} \Gamma(1/2)) = 4/\pi = 1.27324 \dots$. 
% F[n_]:={Sum[2 Sin[Pi k/n],{k,1,n-1}], Im[ 4/(1-Exp[I Pi/n]) ]}/n; F[100] //N
\item For $s=-2$ we have $\sum_{k=1}^{n-1} 2^2 \sin^2(\pi k/n) = 2n$. 
% \sum_{k=1}^{n-1} (1-\cos(2 \pi k/n))/2 = (n+1)/2 - Re(1/(1-e^{i \pi/n}))
% n=47; {Sum[2^2 Sin[Pi k/n]^2,{k,0,n-1}], (2n+1)-Re[ 2/(1-Exp[I Pi/n]), 2n ]}
% Re[ 2/(1-Exp[I Pi/n]) ] == 1  
\item For $s=1$, we see convergence $\zeta_n/(n \log(n)) \to 0.321204$. 
Since $c(1)=\infty$ this agrees in a limiting case with the convergence of 
$(\zeta_n-n c(1))/n = \zeta_n/n - c(1)$.
% n=1000000; N[Sum[N[1/(2 Sin[Pi k/n]),20],{k,1,n-1}]/(n Log[n]),20]
\item For $s=2$, we deal with the discrete analogue of the {\bf Basel problem} and the sum 
$\zeta_n(2) = (1/4) \sum_{k=1}^{n-1} 1/\sin^2(\pi k/n)$. We see that 
$\zeta_n(2)/n^2 \to (1-n^2)/12$ which implies $\zeta(2) = \pi^2/6$. 
% n=100000; Sum[1/(4 Sin[Pi k/n]^2),{k,1,n-1}]/n^2
\item For $s=3$, we numerically computed $\zeta_n(3)/n^3$ for $n=10^7$ as 
$0.009692044$. Since no explicit expressions are known for $\zeta(3)$, it is unlikely 
also that explicit expressions for $\zeta_n$ will be available in the discrete. 
% n=100000; N[Sum[N[1/(8 Sin[Pi k/n]^3),20],{k,1,n-1}]/n^3,20]
\item For $s=4$, we will derive the value $1/720$. Also this will reestablish the known value
value $\zeta(4) = \pi^2/90$ for the Riemann zeta function. 
% n=100000; N[Sum[N[1/(16 Sin[Pi k/n]^4),20],{k,1,n-1}]/n^4,20]
\end{itemize}

\section{The Basel problem}

The finite analogue of the Basel problem is the quest to find an explicit formula for 
$$  A(n) = \frac{\zeta_n(2)}{n^2} = \frac{1}{n^2} \sum_{k=1}^{n-1} \sin^{-2}(\frac{\pi k}{n})  \; . $$ 
Curiously, similar type of sums
have been tackled in number theory by Hardy and Littlewood already \cite{HL,HL46}
and where the sums are infinite over an irrational rotation number. We have 

\begin{lemma}
$\frac{1}{n^2} \sum_{k=1}^{n-1} \sin^{-2}(\frac{\pi k}{n}) = \frac{1-1/n^2}{3}$. 
\end{lemma} 

\begin{proof}
The function $g(x) = \cot(\pi x)$ is a fixed point of the Birkhoff renormalization
$$ T(g)(y)=\frac{1}{n} \sum_{k=0}^{n-1} g(\frac{y}{n} + \frac{k}{n}) \; . $$
Therefore, $g' = -\pi \sin^{-2}(\pi x)$ or $\sin^{-2}(\pi x)$ is a fixed point of 
$$ T_2(g)(y)=\frac{1}{n^2} \sum_{k=0}^{n-1} g(\frac{y}{n} + \frac{k}{n}) \; . $$
Lets write this out as 
$$ \frac{1}{n^2} \sum_{k=1}^{n-1} g(\frac{y}{n} + \frac{k}{n})  = g(y) - 1/n^2 g(\frac{y}{n})  \; . $$
The right hand side is now analytic at $0$ and is $(1-1/n^2)/3 + A_2 x^2  + A_4 x^4 \dots $.
Therefore, the limit $y \to 0$ gives 
$$  \sum_{k=1}^{n-1} \frac{1}{n^2} \sin^{-2}(\pi k/n) = (1-\frac{1}{n^2})/3 \; . $$
% n = 3; Plot[ (Sin[Pi x]^{-2} - Sin[Pi x/n]^(-2)/n^2) - (1 - 1/n^2)/3,
The following argument just gives the limit $A$. Chose $n=3m$ and chose three sums for $y=0,y=1/3$ and $y=2/3$: 
\begin{eqnarray*}
   \frac{1}{m^2} [  g(\frac{1}{n}) + g(\frac{4}{n}) + \cdots + g( \frac{n-3}{n} ) ] &\to& A \\
   \frac{1}{m^2} [  g(\frac{2}{n}) + g(\frac{5}{n}) + \cdots + g( \frac{n-2}{n} ) ] &\to& g(1/3) = 4/3   \\
   \frac{1}{m^2} [  g(\frac{3}{n}) + g(\frac{6}{n}) + \cdots + g( \frac{n-1}{n} ) ] &\to& g(2/3) = 4/3   \;, 
\end{eqnarray*} 
where the limits are $n=3m \to \infty$. When adding this up we get $A 3^2$ on the left hand side and $A + 4/3 + 4/3$ on the
right hand side in the limit. From $A (3^2-1)  = 2 (4/3)$ we get $A=1/3$. 
\end{proof}

{\bf Example.} For $n=2$ we get $\zeta_9(2) = 20/3$. The scalar Laplacian which agrees with the
Laplacian on $1$-forms is displayed in the introduction. 
The nonzero eigenvalues are $3.87939, 3.87939, 3, 3, 1.6527, 1.6527, 0.467911, 0.467911$. 
The sum of the reciprocals $\sum_i \lambda_i^{-1}$ indeed is equal to $20/3$. This is the 
trace of the Green function. Note that since in the circle case, the scalar Laplacian $L_0$ and the one form Laplacian 
$L_1$ are the same since we have chosen only the positive part of the Dirac operator $D$ which 
produces $L=D^2 = L_0 \oplus L_1$, we
have $\zeta(2k) = {\rm Tr}(L^{-k})/2$, where ${\rm Tr}$ is the sum of all nonzero eigenvalues. 
We had to divide by $2$ because we only summed over the positive eigenvalues of $D$. \\

The formula $\zeta_n(2) = (n^2-1)/12$ is true for all $n$:

\begin{thm}[Discrete Basel problem]
For every $n \geq 1$, 
$$ \zeta_n(2) = (n^2-1) \frac{1}{12} \; . $$ 
It is a solution of all the arithmetic recursions
$$  A(m) + A(k) k^2 = A(m k) k^2 \; . $$
\end{thm} 
\begin{proof} 
The previous derivation for $m \to 3m$ can be generalized to $m \to k m$. 
It shows that $A(m) + A(k) k^2 = A(m k) k^2$, which is solved by $(1-1/n^2)/12$. 
\end{proof} 

{\bf Remarks.} \\
{\bf 1)} $B(n)= n^2 A$ satisfies the relation $B(m)/m^2 + B(k) = B(m k)/m^2$. 
$C(n)= n A$ satisfies the relation $C(m) + C(k) k m = C(m k) k$. \\
{\bf 2)} The answer to $\zeta_n(2) = {\rm tr}(L^{-1})$ 
is the variance of uniformly distributed random variables in $\{1, \dots ,n \; \}$:
$(1/n) \sum_{k=1}^n k^2 - ((1/n) \sum_{k=1}^n k)^2=(n^2-1)/12$. Is this an
accident? 
% Sum[k^2,{k,1,10}]/10 - (Sum[k,{k,1,10}]/10)^2 == (10^2-1)/12

\begin{coro}[Classical Basel problem] 
$\zeta(2) = \sum_{k>0} \frac{1}{k^2} = \frac{\pi^2}{6}$. 
\end{coro}
\begin{proof} 
By Theorem~(\ref{maintheorem}), we have
$\zeta(s) = \lim_{n \to \infty} \zeta_n(s) 2^s \pi^s/(2 n^s)$. For $s=2$, we get from 
the discrete Basel problem the answer $\lim_{n \to \infty} (1-1/n^2) \pi^2/6 = \pi^2/6$. 
\end{proof} 

Lets look at the next larger case and compute $\tr(L^{-4})$:  

\begin{lemma}
For every $n$,
$$ \sum_{k=1}^{n-1} \sin^{-4}(\frac{\pi k}{n}) = \frac{(n^2-1)(n^2+11)}{45} \; . $$
%  n=12; Sum[ Sin[Pi k/n]^(-4),{k,1,n-1}] == (n^2-1)(n^2+11)/45
In the limit this gives $\frac{1}{45}$ so that
$\sum_{k>0} \frac{1}{k^4} = \frac{\pi^4}{90}$. 
\end{lemma}
\begin{proof}
a) We can get an explicit formula for the sum $\sum_{k=1}^{n-1} (2+\cos(2\pi k/n))/\sin^4(\pi k/n)$
by renormalization. We can also get an explicit formula for 
$$ \sum_{k=1}^{n-1} (2-2\cos(2\pi \frac{k}{n}))/\sin^4(\pi \frac{k}{n}) = \sum_{k=1}^{n-1} \sin^{-2}(\pi \frac{k}{n})  $$
which allows to get one for $\sum_{k=1}^{n-1} 1/\sin^4(\pi k/n)$.  \\
While b) follows from a), there is a simpler direct derivation of b) which works in the case when we 
are only interested in the limit. 
With have $g'(x) = -\pi\sin^{-2}(\pi x)$, then $g'''(x) = -\pi^3 2(2+\cos(2\pi x))/\sin^4(\pi x)$ 
which has the same singularity at $x=0$ than $-\pi^3 6 \sin^{-4}(x)$. 
Now $h=g'''(x)$ is a fixed point of the Birkhoff renormalization operator
$$ T_4(g)(y)=\frac{1}{n^4} \sum_{k=0}^{n-1} g(\frac{y}{n} + \frac{k}{n}) \; . $$
Lets call the limit $A$. As before, we see that the limit $A$ satisfies 
$A+h(1/3)+h(2/3) = A 3^4$. Since $h(1/3)=h(2/3)=\pi^3 16/3$. We have $A=\pi^3 (32/3)/(3^4-1) = \pi^3 2/15$. 
If we do the summation with $\sin^{-4}(x)$ we get $(1/6 \pi^3)$ that result which is $1/45$.  \\
\end{proof}

\section{Proof of the limiting root result}

In this section we prove the statements in Theorem~(\ref{maintheorem}). 
We will use that if a sequence $f_n$ of analytic functions converges to $f$ uniformly 
in an open region $G$ and $K$ is a subset of $G$ and all $f_n$ have a root in $K$. 
Then $f$ has a root in $K$. This is known under the name Hurwitz theorem in complex
analysis (e.g. \cite{Conway1978} and follows from a contour integral argument. \\

{\bf a) Below the critical strip.}
The first part deals with non-positive values of $\sigma$, where $\zeta$ is a Riemann sum.

\begin{propo}
For large enough $n$, there are no roots of $\zeta_n(s)$ in $\sigma \leq 0$. 
\end{propo}
\begin{proof}
Below the critical strip $\sigma \leq 0$, the sum $\zeta_n(s)/n$ is a Riemann sum $\sum_{k=1}^{n-1} g(k/n)$
for a continuous, complex valued function $g(x) = \sin^{-s}(x)$ on $[0,1]$. 
The function $c(s)$ has no roots in $\sigma \leq 0$. By Hurwitz theorem, also $\zeta_n$ has
no roots there for large enough $n$. 
\end{proof} 

There are still surprises. We could show for example that $[\zeta_n(s)-n c(s)]/n^s$ even converges
even for $-1<s<0$. This shows that the Riemann sum $\zeta_n(s)/n$ is very close to the average $c(s)$
below the critical strip. We do not need this result however. \\

{\bf To the right of the critical strip} \\
There is a relation between the zeta function of the circular
graphs $C_n$ and the zeta function of the circle $T^1$ if $s$ is to the right of the
critical line $\sigma=1$. Establishing this relation requires a bit of analysis.  \\

By coincidence that the name {\bf ``Riemann"} is involved 
both to the left of the critical strip with Riemann sum and to the right with 
the Riemann zeta function. The story to the right of the critical line is not
as obvious as it was in the left half plane $\sigma<0$. 

\begin{propo}[Discrete and continuum Riemann zeta funtion]  % AAA
If $\zeta_n(s)$ is the zeta function of the circular graph $C_n$ and $\zeta(s)$ is
the zeta function of the circle (which is the classical Riemann zeta function), then
$$   \zeta_n(s) \frac{2 (2\pi)^s}{n^s} \to \zeta(s) $$
for $n \to \infty$. The convergence is uniform on compact sets $K \subset \{ \sigma>1 \; \}$
and $\zeta_n$ has no roots in $K$ for large enough $n$. 
\end{propo}
\begin{proof}
We have to show that the sum 
$$ \sum_{k=1}^{n-1} [ 2^{-s} \sin^{-s}(\pi \frac{k}{n}) \frac{(2\pi)^s}{n^s} - \frac{1}{k^s} ] $$
goes to zero. Instead, because the Riemann sums $\sum_{k=1}^{n-1} g(k/n)$ of $g(x)=1/x$ and $g(x) = 1/(1-x)$ give the same result, 
we can by symmetry take twice the discrete Riemann sum and look at the Riemann sum 
$$  \frac{S_n}{n^s} $$ with $S_n = \sum_{k=1}^{n-1} h(k/n)$ and 
$$ h(x) = 2 2^{-s} \sin^{-s}(\pi x) - 2^{-s} (\pi x)^{-s} \; . $$
We have to show that $S_n/n^s$ converges to zero.  This is equivalent to show
that $n^{-(s-1)} (S_n/n)$ converges to zero. Now, 
$$ \frac{S_n}{n} = H(1-\frac{1}{2n})-H(\frac{1}{2n}) + \frac{1}{n} \sum_k K(\frac{k}{n})  \; , $$
where $H(x)$ is the anti-derivative of $h$ and where $K(x)$ is a $K$-derivative of $h$. 
Since $\zeta$ has no roots in $K$, we know by Hurwitz also that $\zeta_n$ has no roots for large enough $n$. 
\end{proof}

{\bf On the critical line} \\
On the critical line, we have a sum 
$$  \zeta(i \tau) = \sum_{k=1}^{n-1} \frac{1}{\sin(\pi \frac{k}{n})} e^{-i \tau \log(\sin(\pi \frac{k}{n})} \; . $$
Now $\sin^{-1}(\pi x)$ is no more in $L^1$ but we have 
$\int_{1/n}^{(n-1)/n} \sin^{-1}(x) \; dx \sim c_n = -2 \log(n) + C_n$.
This shows that $(\zeta_n - C_n)/(n \log(n))$ plays now the same role than $(\zeta_n(s) - c(s))/n^s$ before. 
This only establishes that $\zeta(s)/(n \log(s))$ stays bounded but does not establish convergence.
We measure experimentally that $\zeta(i \tau)/(\log(n) n)$ could converge to a bounded function which 
has a maximum at $0$. Indeed, we can see that the maximum occurs at $\tau=0$ and that
the growth is logarithmic in $n$. 
We do not know whether there are roots of  $\zeta_n(s)$ accumulating for $n \to \infty$ on $\sigma=1$.  \\

{\bf Inside the critical strip} \\
The following result uses Rolle theory developed in the appendix:

\begin{thm}[Central limit in critical strip]
For $0 < \sigma={\rm Re}(s) < 1$, the limit 
$$ \chi(s) = \lim_{n \to \infty} \frac{\zeta_n(s) - n c(s)}{n^s} $$
exists. There are no roots in the open critical strip.
\end{thm}
\begin{proof}
On the interval $[1/n,(n-1)/n]$, the function $\sin(\pi x)/n^s$ takes values between
$[-1,1]$ and subtracting $n c(s)/n^s$ gives a function $h$ on $[0,1]$ 
which averages to $0$. We have now a smooth bounded function $h$ of zero average
and have to show that $\sum_{k=1}^{n-1} h(k/n)$ converges. Let $H$ be the antiderivative,
Now use the mean value theorem to have $\sum_{k=1}^{n-1} H'(y_k) =0$, where $y_k$ are the 
Rolle points satisfying $H'(y_k) = h((k+1)/n) - h(k/n)$.
deviation of the Rolle point from the mean. We will see that 
$$ \frac{\zeta_n(s) - n c(s)}{n^s} = \frac{1}{n^{2+s}} \sum_{k=1}^{n-1} K(\frac{k}{n}) \sin^{-2-s}(\pi \frac{k}{n}) \; . $$
We know that $[\zeta_n(z_n)-n c(z_n)]/n^s$ converges for $n \to \infty$ uniformly for $s$
on compact subsets of the strip $0<\sigma<1$. Hurwitz assures that there are no roots in the limit. 
\end{proof}
% a = NIntegrate[1/Sqrt[Sin[Pi x]], {x, 0, 1}];
% n = 10000000; Sum[ N[(Sin[Pi (k/n)]^(-1/2) - a)/Sqrt[n]], {k, 1, n - 1}]

{\rm Remarks.} \\
{\bf 1)} We actually see convergence for $-2 < \sigma < 1$. \\
{\bf 2)} The above proposition
does not implie that $\zeta_n(s)$ has no roots outside the critical strip for large enough $n$. 
Indeed, we believe that there are for all $n$ getting closer and closer to the critical line. \\
{\bf 3)} Because of the almost periodic nature of the setup we believe that for every $\epsilon>0$
there is a $n_0$ such that for $n>n_0$, there are no roots in $\sigma \in [-\epsilon,1+\epsilon]$. 

\begin{figure}
\scalebox{0.03}{\includegraphics{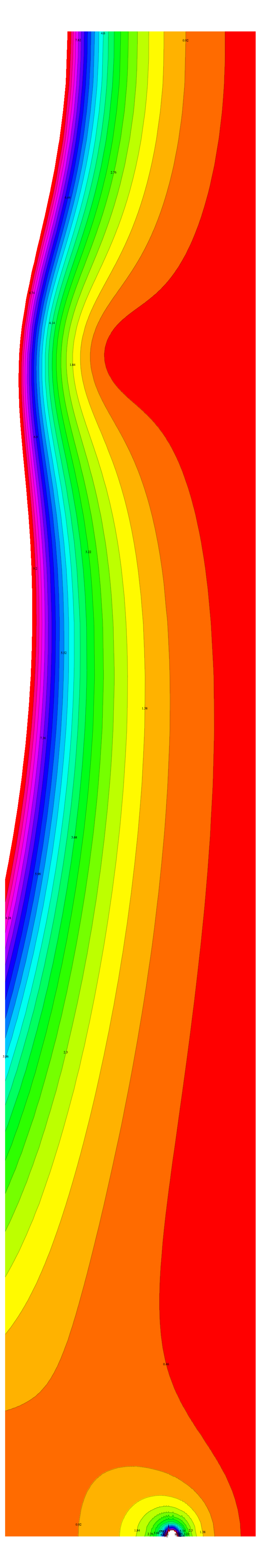}}
\caption{
The limiting function  $\chi(s)$ shown in the complex plane on $-1<\sigma<2$ and $0<\tau<18$
as a contour plot of $|\chi(s)|$. 
}
\label{situation}
\end{figure}

\section{Questions}

{\bf A)} This note has been written with the goal to understand the roots of $\zeta_n$  and
to explain their motion as $n$ increases. We still do not know:

\parbox{12cm}{ A) What is the limiting set the roots approach to on the critical line?  }

We think that it the limiting accumulation set could coincide with the entire critical line. If this is not true, where are the gaps?  \\

{\bf B)} The zeta functions considered here do not belong to classes of Dirichlet series 
which are classically considered. The following question is motivated from the classical symmetry for 
$\xi(s) = \pi^{-s/2} s(s-1) \Gamma(s/2) \zeta(s)$ which 
satisfies $\xi(1-s)=\xi(s)$ and where $S(s)=1-s$ is the involution:

\parbox{12cm}{B) Is there for every $n$ a functional equation for $\zeta_n$. }

We think the answer is yes in the limit. A functional equation which 
holds for each $n$ would be surprising and very interesting. \\

{\bf C)} The shifted zeta function
$$ \zeta_n(y,s) = \sum_{k=1}^{n-1} (2 \sin(\pi (k+y)/n))^{-s} \; . $$
is different from the ``twisted zeta function' 
$\zeta(s) = \sum_{\lambda} (\lambda+a)^{-s}$
which motivated by the Hurwitz zeta function (\cite{Dowker} for some integer values of $s$).
The function 
$$  z_n(s,y) = \zeta_n(s,0) + \zeta_n(s,y)  $$ 
has the property that $z_{n}(s,0) = 2\zeta_n(s), z_n(s,1/2) = \zeta_{2n}(s) + (2 \sin(\pi/n))^{-s}$. 
We can now define a vector field 
$$  F_n(s)=-\partial_s z_n(s,y)/\partial_h z_n(s,y) |_{y=0} $$
which depends on $n$ and has the property that the level set $s=c$ moves with $y$. This applies especially
to roots. If we start with a root of $\zeta_n$ at $y=0$, we end up with a root of $\zeta_{2n}$ 
at $y=1/2$. This picture explains why the motion of the roots has a natural unit time-scale
between $n$ and $2n$ but it also opens the possibility that we get a limiting field: 

\begin{comment}
zeta[n_,s_,y_]:=Sum[ (Sin[Pi (k+y)/n])^(-s),{k,1,n-1}]; 
zeta[32,s,0] - (zeta[16,s,0] + zeta[16,s,1/2])
\end{comment}

\parbox{12cm}{C) Do the fields $F_n$ converge to a limiting vector field $F$? }

We expect a limiting vector field $F$ to exist in the critical strip and that $F$ is 
not zero almost everywhere on the critical line. This would show that there are no
gaps in the limiting root set. 

\section{Tracking roots} 

We report now some values of the root closest to the origin. These computations were done in 
the summer of 2013 and illustrate how slowly the roots move when $n$ is increased.
The results support the picture that the distance between neighboring roots will go to zero for 
$n \to \infty$ even so this is hard to believe when naively observing the numerics at first. For integers 
like $n=10'000$, we still see a set which appears indistinguishable from say $n=10'100$, giving the 
impression that the roots have settled. However, the roots still move, even so very slowly:
even when observing on a logarithmic time scale, the motion will slow down exponentially. Their speed
therefore appears to be asymptotic to $A e^{-C e^t}$ for some constants $A,B$ which makes their
motion hard to see. For the following data, we have computed the first
root with $400$ digit accuracy and stopped in each case the Newton iteration when 
$|\zeta_n(z)|<10^{-100}$. 

\begin{center}
\begin{tabular}{ll}
n         &   root $r(n)$     \\ \hline
$ 5000$   &   $z = 1.0147497138 + 0.7377785810 i$  \\
$10000$   &   $z = 1.0120939324 + 0.6811471384 i$  \\
$20000$   &   $z = 1.0100259045 + 0.6327440122 i$  \\
$40000$   &   $z = 1.0083949013 + 0.5908710693 i$  \\
$80000$   &   $z = 1.0070933362 + 0.5542725435 i$  \\
$160000$  &   $z = 1.0060433506 + 0.5219986605 i$  \\
$320000$  &   $z = 1.0051878321 + 0.4933170571 i$  \\
$640000$  &   $z = 1.0044843401 + 0.4676534598 i$  \\
$1280000$ &   $z = 1.0039009483 + 0.4445508922 i$  \\
$2560000$ &   $z = 1.0034133623 + 0.4236409658 i$  \\
$5120000$ &   $z = 1.0030028952 + 0.4046232575 i$  \\
$10240000$&   $z = 1.0026550285 + 0.3872502227 i$  \\
$20480000$&   $z = 1.0023583755 + 0.3713159790 i$  \\ \hline
\end{tabular}
\end{center}

If $r(n) \in C$ denotes  the first root of $\zeta_n(s)$ , then linear interpolation of the ${\rm Re}(r(n))$ data
indicate that for $n=10^{10}$ only, we can expect the root to be in the strip. This
estimate is optimistic because the root motion will slow down. It might 
well be that we need to go to much higher values to reach the critical strip, or 
- which is very likely by our result - that the critical strip
is only reached asymptotically for $n \to \infty$. A reason for an even slower convergence
is also that $\zeta(s)$ has a pole at $s=1$, which could be the final landing point of the 
first root we have tracked. 

\begin{figure}
\scalebox{0.13}{\includegraphics{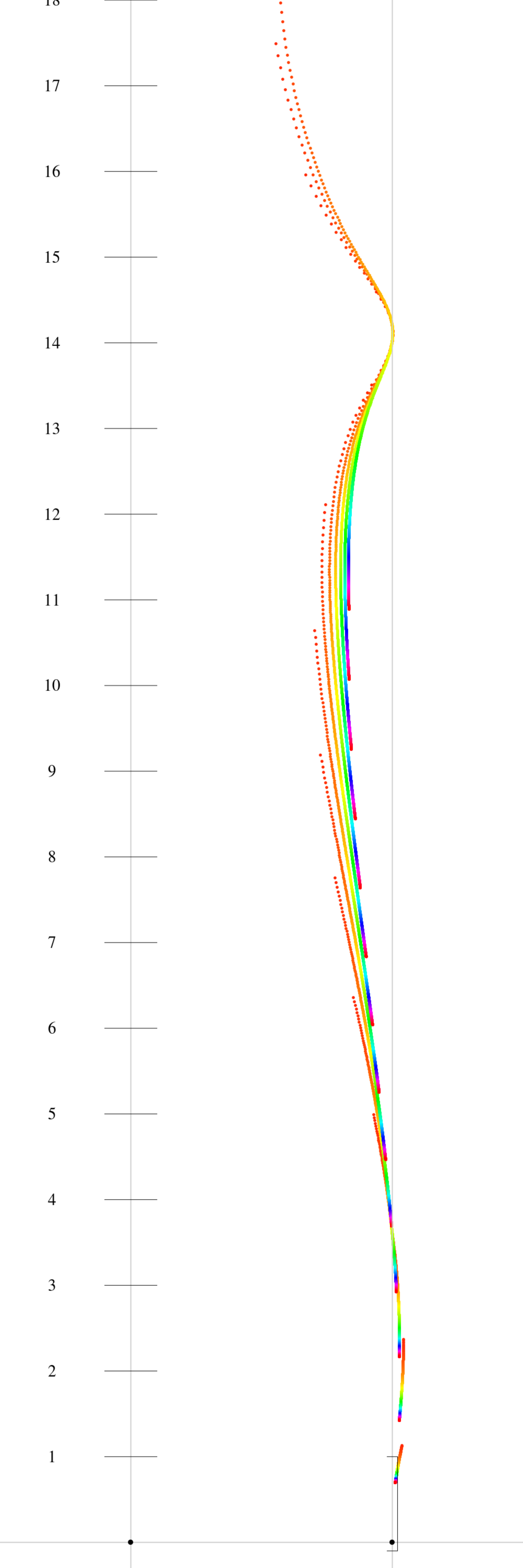}}
\scalebox{0.13}{\includegraphics{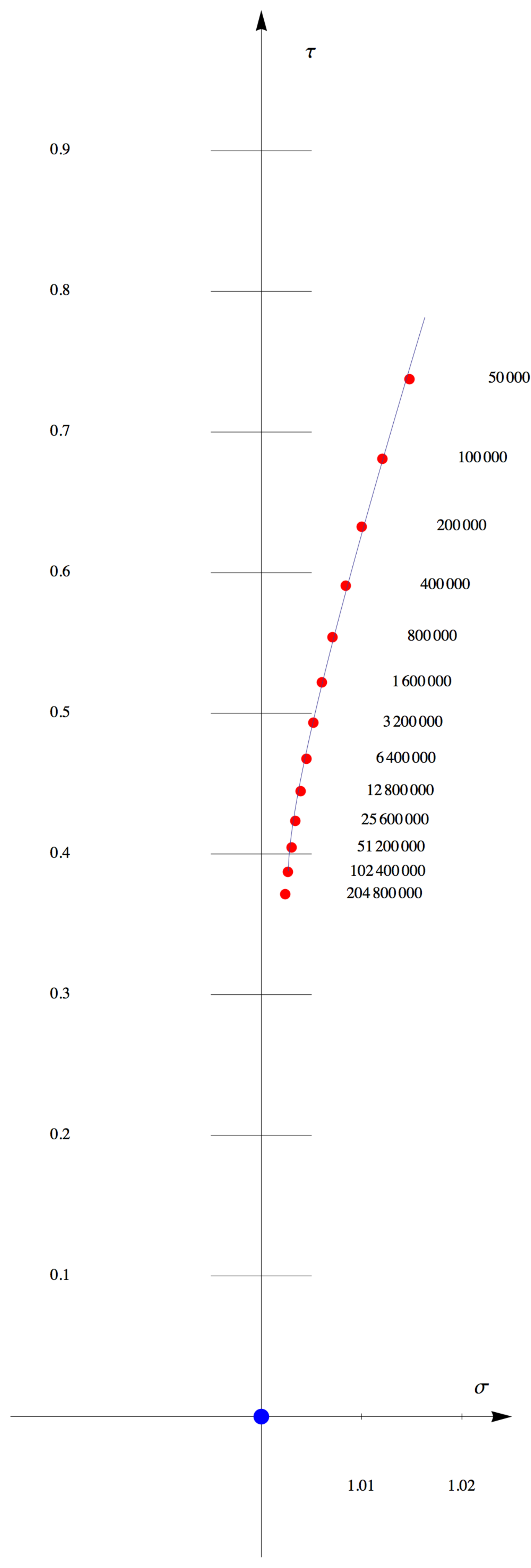}}
\caption{
Tracking the roots of $\zeta_n(s)$ for $n=100$ to $n=6000$ 
shows that their paths appear to move smoothly. 
There are still roots outside the strip but they will eventually
move to the right boundary $\sigma=1$ of the critical strip $0<\sigma<1$. 
We see that the root motion line touches the line $\sigma=1$. 
The right figure shows the tiny rectangle at the lower right corner of the 
first picture, where we can see the first root, tracked in an expensive long run
for $n=5000,10000,\dots, 20480000$ and interpolated with a quadratic fit.
The vertical line is $\sigma=1$, which contains the pole $s=1$. 
}
\label{situation}
\end{figure}

\begin{figure}
\scalebox{0.03}{\includegraphics{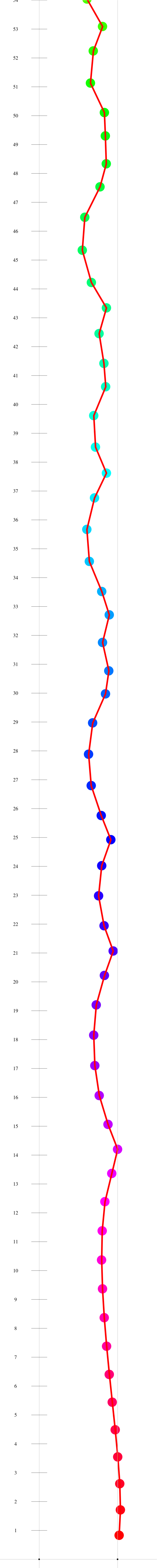}}
\scalebox{0.03}{\includegraphics{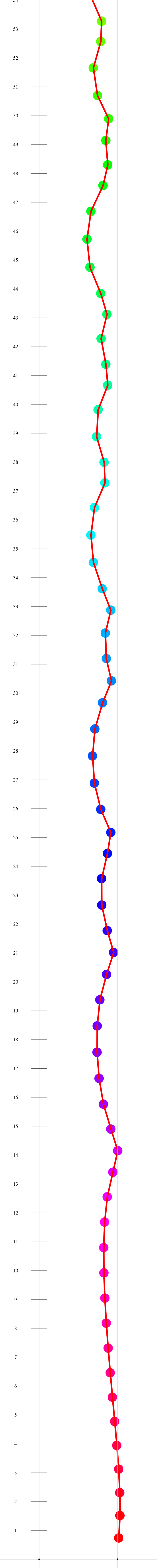}}
\scalebox{0.03}{\includegraphics{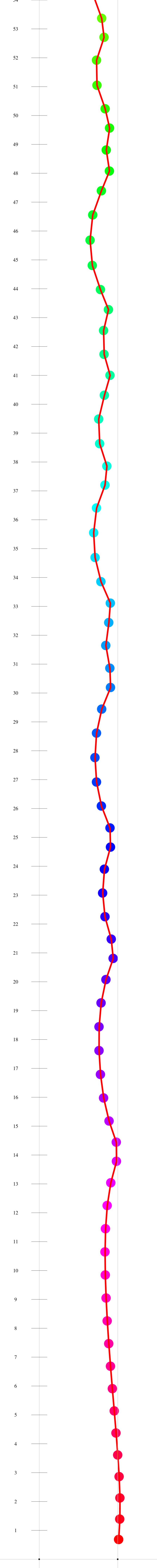}}
\scalebox{0.03}{\includegraphics{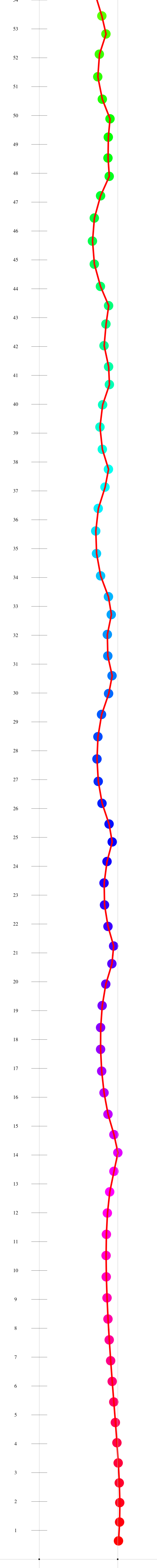}}
\scalebox{0.03}{\includegraphics{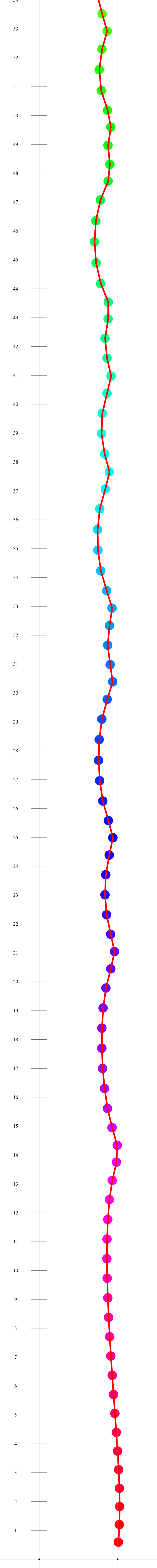}}
\scalebox{0.03}{\includegraphics{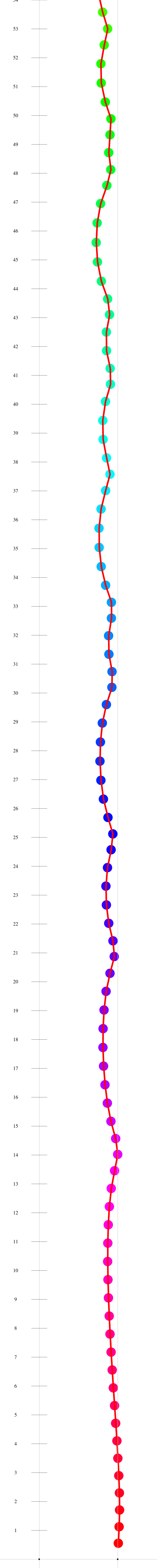}}
\caption{
The roots of $\zeta_n(s)$ for $n=2000,5000$, then for 
for $n=10000,n=20000,n=40000$ and finally for $n=80000$
are shown on the window $-0.5 \leq \sigma \leq 1.5, 0 \leq \tau \leq 54$.  
}
\label{situation}
\end{figure}

\begin{figure}
\scalebox{0.17}{\includegraphics{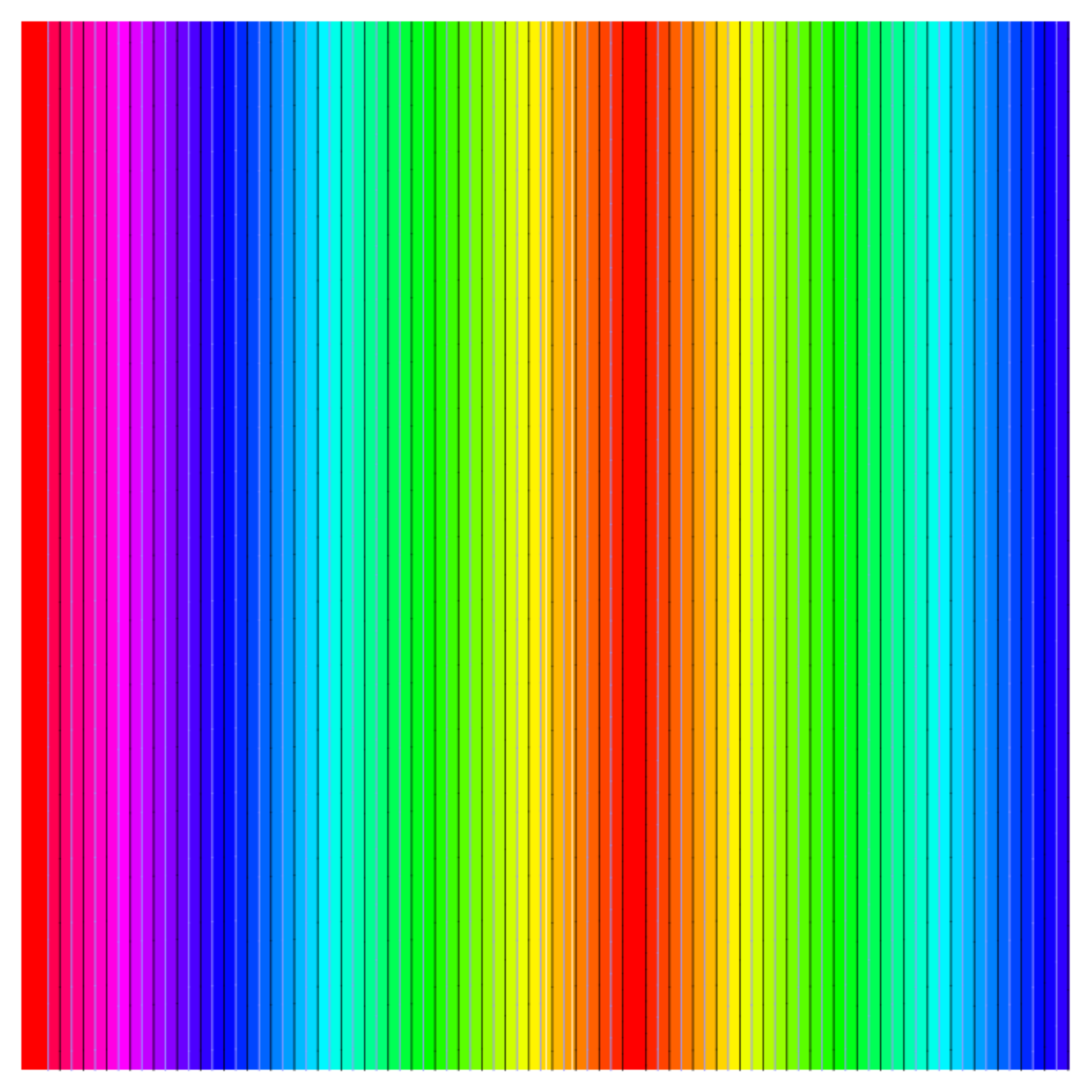}}
\scalebox{0.17}{\includegraphics{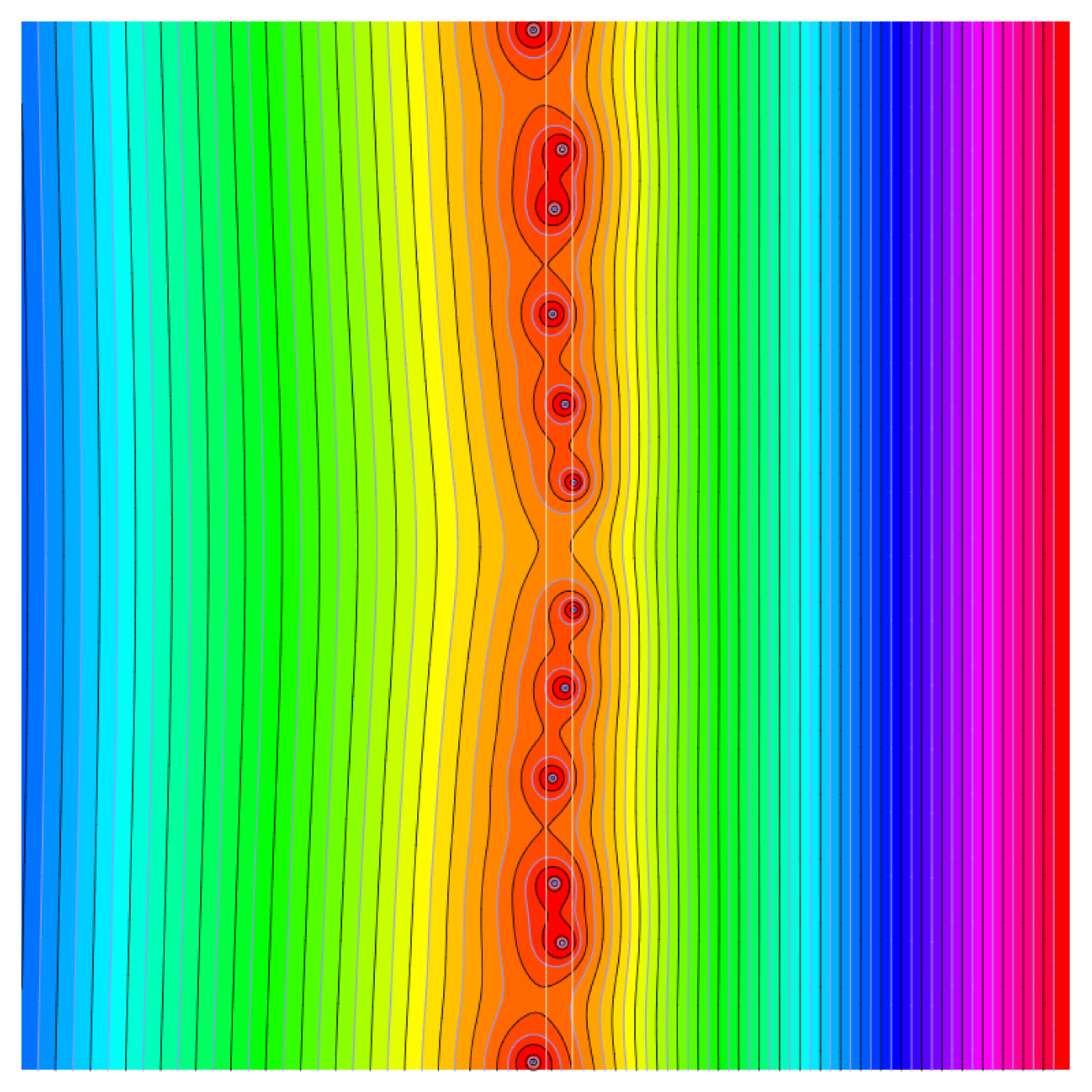}}
\scalebox{0.17}{\includegraphics{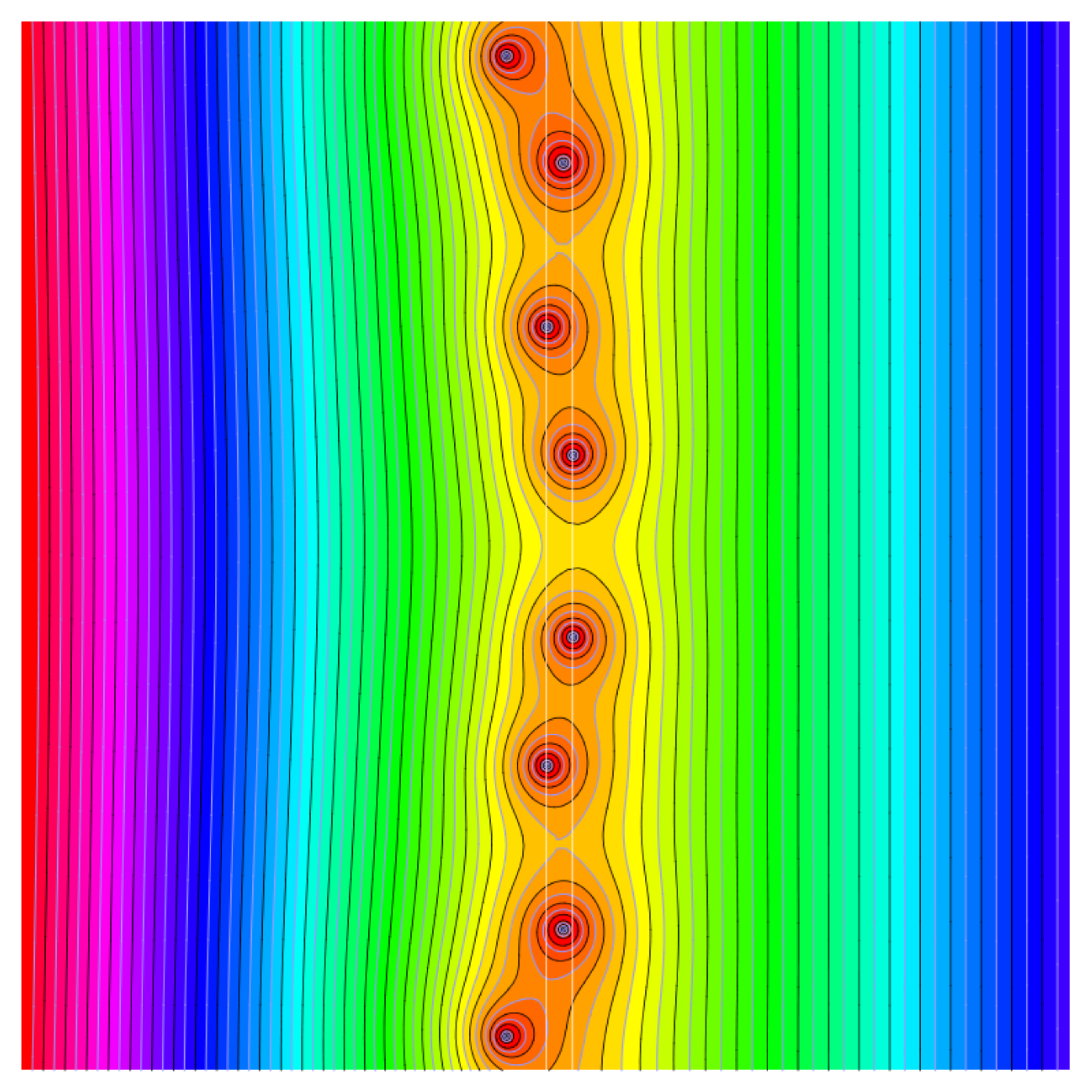}}
\scalebox{0.17}{\includegraphics{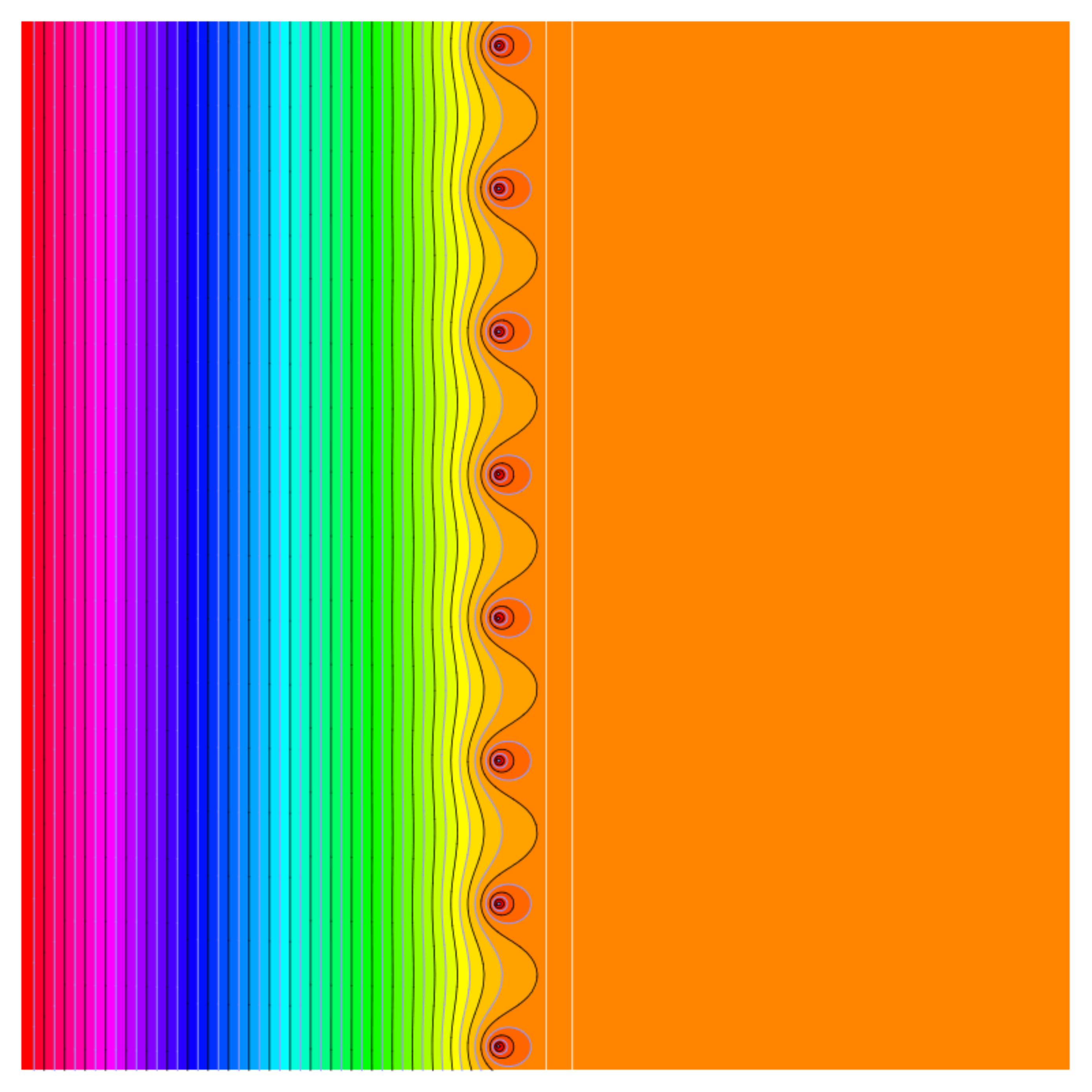}}
\scalebox{0.17}{\includegraphics{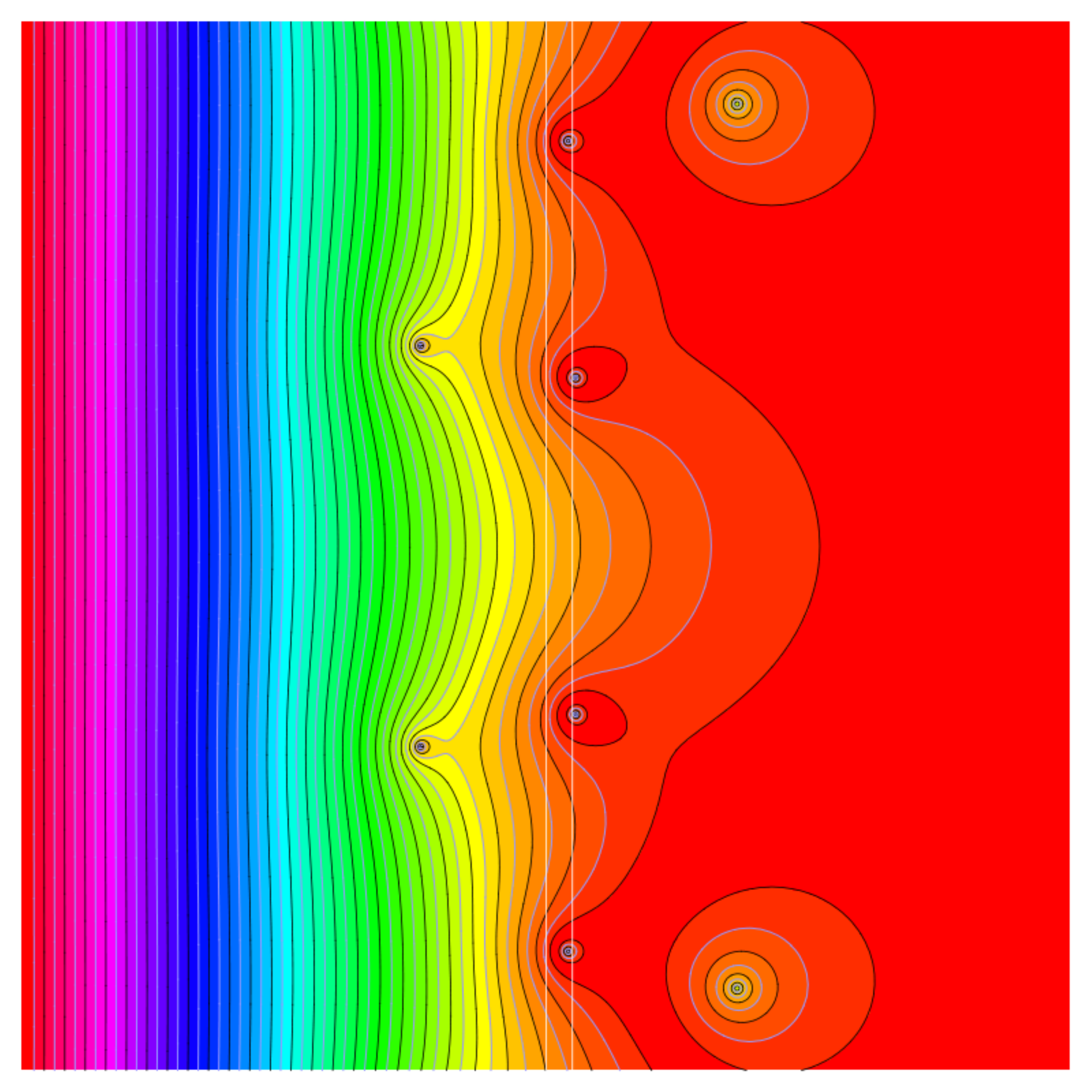}}
\scalebox{0.17}{\includegraphics{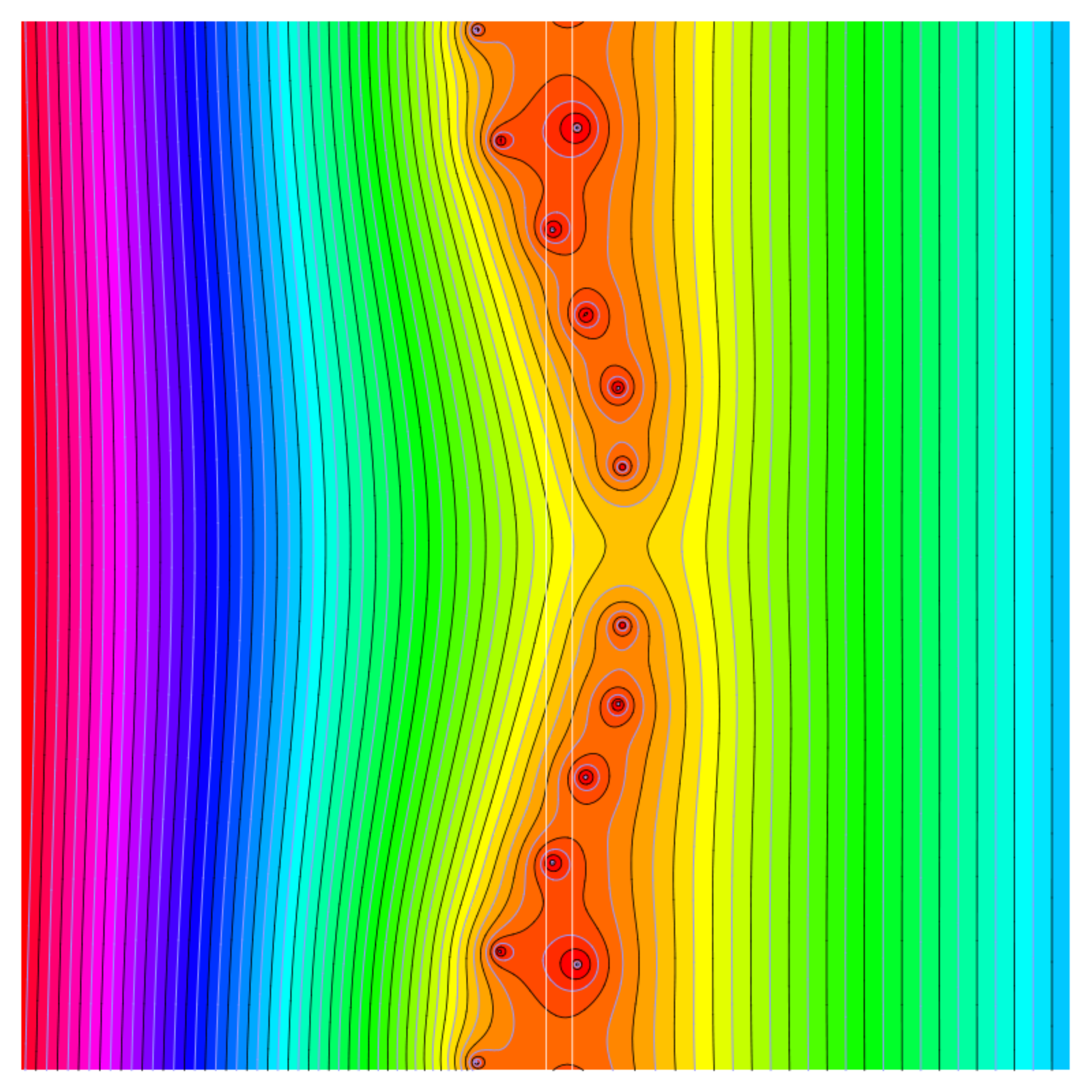}}
\scalebox{0.17}{\includegraphics{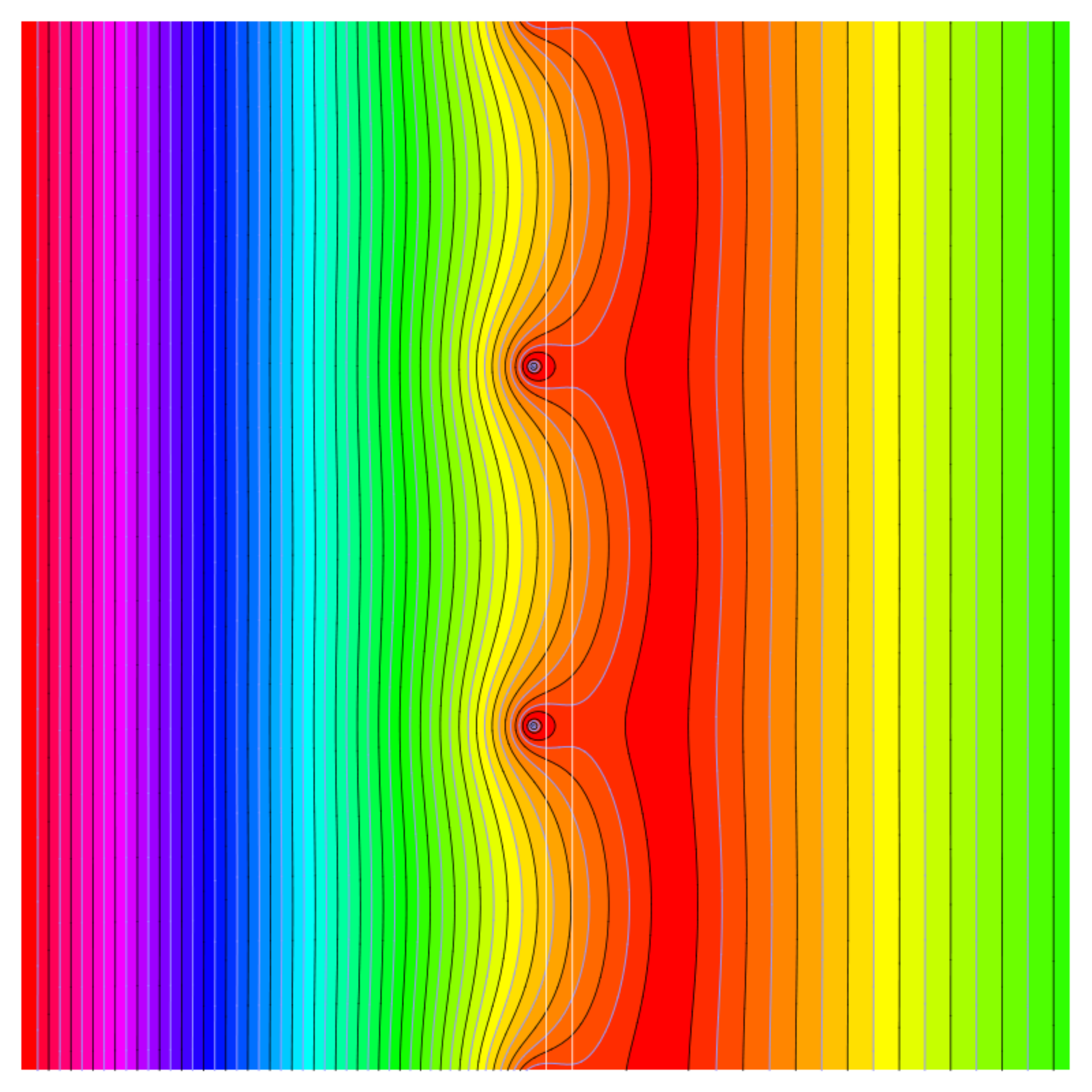}}
\scalebox{0.17}{\includegraphics{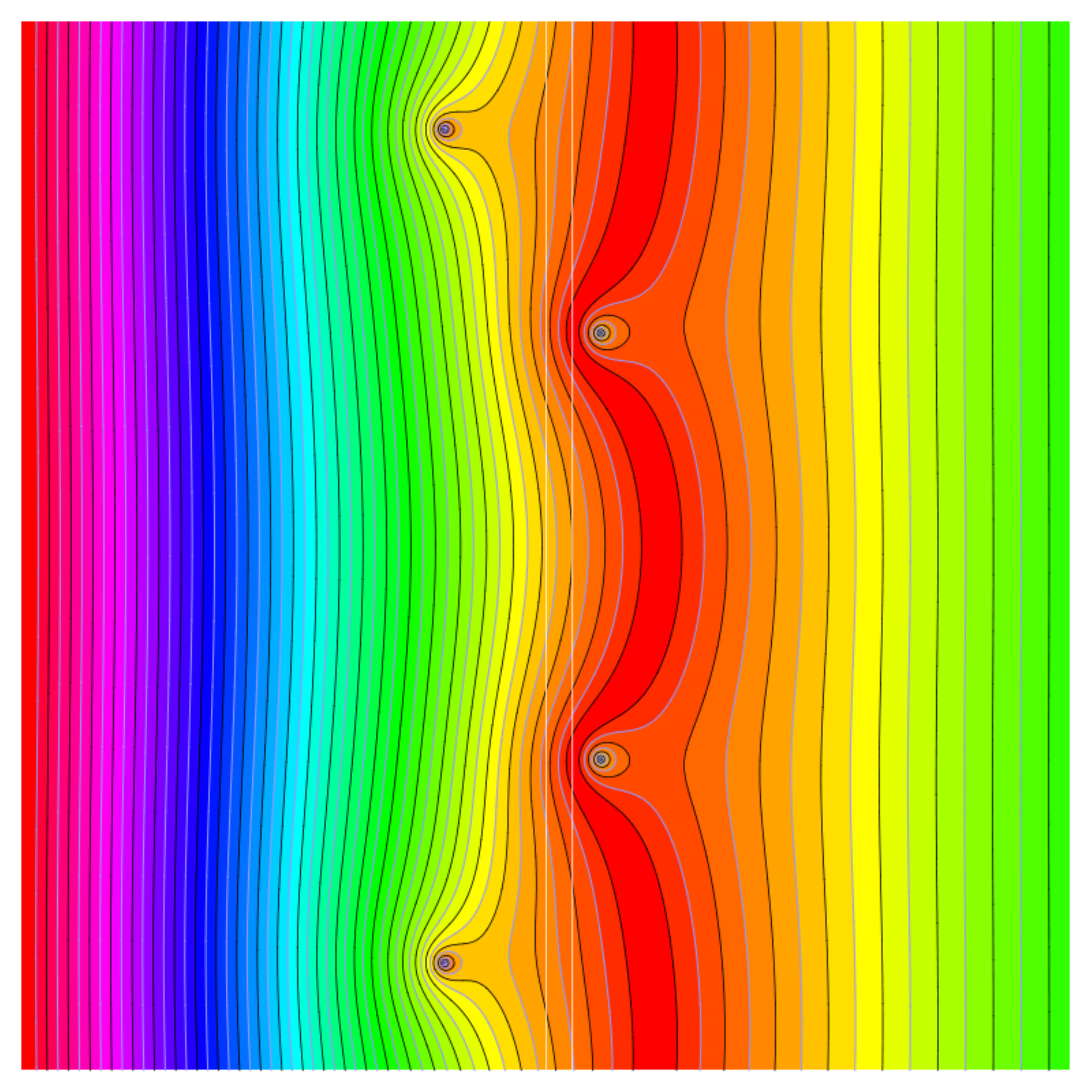}}
\caption{
The zeta functions for the 
complete graph $K_5$ the linear graph $L_{10}$ the circular graph 
$C_{10}$ the star graph $S_{10}$ the wheel graph $W_{10}$, a random Erdoes-R\'enyi 
graph $E_{20,0.5}$, the Peterson graph $P_{5,2}$ and the octahedron graph $O$.  }
\label{diracgraphexamples}
\end{figure}

\section{Appendix: Newton-Cotes with singularities}

This appendix contains a result which complements the Newton-Cotes formula for 
functions which can have singularities at the boundary of the integration 
interval. It especially applies for functions like $\sin(\pi k x)^{-s}$ which 
appear in the zeta function for discrete circles. We might expand the theme covered in this appendix 
elsewhere for numerical purposes or in the context of central limit theorems for 
stochastic processes with correlated and time adjusted random variables. Both probability 
theory and numerical analysis are proper contexts in which {\bf Archimedean 
Riemann sums} - Riemann sums with equal spacing between the intervals - can be considered. \\

We refine here Newton-Coates using a derivative we call the 
{\bf K-derivative}. This notion is well behaved for functions like $f(x) = \sin^{-s}(\pi x)$ on 
the unit interval if $s \in (0,1)$. To do so, we study the deviation of Rolle points 
from the midpoints in Riemann partition interval and use this to improve the estimate
of Riemann sums. If we would know the Rolle points exactly, then the finite Riemann sum would
represent the exact integral. Since we only approximate the Rolle points, we get close. \\

In the most elementary form, the classical {\bf Newton-Cotes formula} for smooth functions 
shows that if $\int_0^1 f(x) \; dx={\rm E}[F]$, then
$$ \lim_{n \to \infty} [\sum_{k=1}^{n-1} f(\frac{k}{n}) -n {\rm E}[f]] $$ 
is bounded above by ${\rm max}_{0 \leq x \leq 1} |f''(x)|/2$. 
This gives bounds on how fast the Riemann sum converges to the integral ${\rm E}[f]$. \\

Since we want to deal with Riemann sums for functions which are unbounded at the 
end points, where the second derivative is unbounded, we use
a {\bf K-derivative} which as similar properties
than the {\bf Schwarzian derivative}. The lemma below was developed when studying
the cotangent Birkhoff sum which features a selfsimilar Birkhoff limit. 
We eventually did not need this Rolle theory in \cite{knillcotangent} but it comes handy here
for the discrete circle zeta functions. As we will see, the result has a probabilistic
angle. For a function $g$ on the unit interval, we can see the Riemann sum $y \to \sum_{k=1}^{n-1} g((k+y)/n)$ 
as a new function which can be seen as a sum of random variables and consider normalized limits
as an analogue of central limits in probability theory. 

\begin{figure}
\scalebox{0.38}{\includegraphics{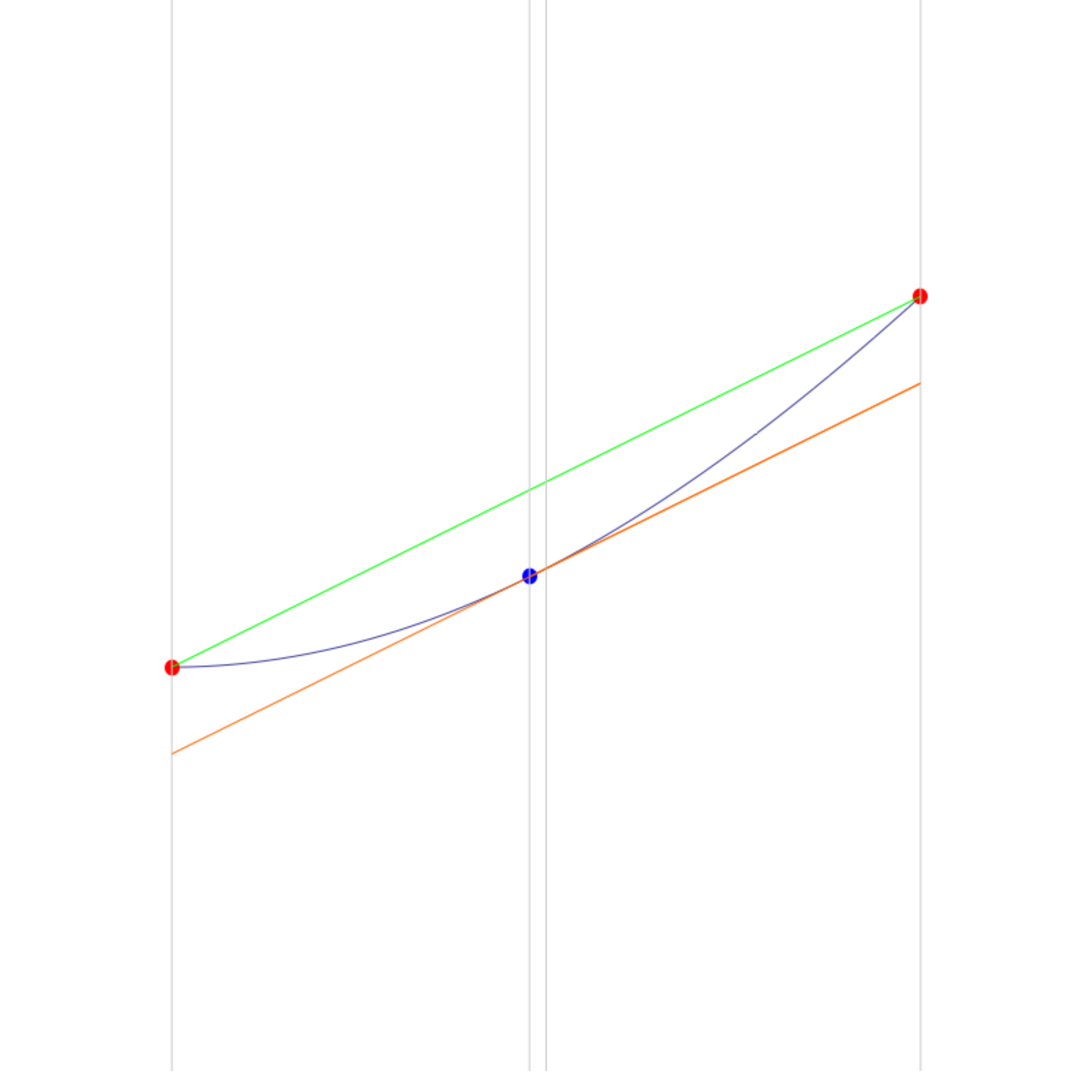}} \\
\caption{
The {\bf Rolle point} of an interval $(-a,a)$ is the point $x(a) \in (-a,a)$ satisfying $f(a)-f(-a)=f'(x(a))$ 
and which is closest to $0$. If there should be two, we could chose the right one to make it unique. 
The existence of Rolle points is assured by the mean value theorem. For small $a$ and nonlinear $f$, it is close to the center $0$ 
of the interval. The picture illustrates, how close the Rolle point $x(a)$ is in general to the
midpoint of an interval. The value $f(x(a))$ is $K_f(0) f''(0) a^2$ close to $f(0)$, 
if $K_f(0)$ is the $K$-derivative. Note that we can have $x(a)$ be discontinuous.}
\end{figure}

When normalizing this so that the mean is $0$ and the variance is $1$, we get a limiting function
which plays the role of entropy maximizing fixed points in probability theory, which is an approach
for the central limit theorems. While for smooth functions, the 
limiting function is linear, the limiting function is nonlinear if there are singularities at the 
boundary. We will discuss this elsewhere.  \\

Given a function $f$ which is smooth in a neighborhood of a point $x_0$, the
{\bf Rolle point} $r(a)$ at $x_0$ is defined as a solution
$x=x(a) \in [x_0-a,x_0+a]$ of the {\bf mean value equation}
\begin{equation}
\label{impliciteequation}
 g(a,x) = f(x_0+a)-f(x_0-a) - 2a f'(x) = 0  \; 
\end{equation}
which is closest to $x_0$ and to the right of $x_0$ if there should be two closest points. 
The derivative $f'(x(x))$ at $x(a)$ agrees with the slope of the line segment connecting 
the graph part on the interval $I=(x_0-a,x_0+a)$. \\

Given Rolle points $r_k$ inside each interval $I_k$ of a partition,
a {\bf discrete fundamental theorem of calculus} states that
$$  \sum_{j=1}^k g'(r_j) \; h_j = g(1)-g(0)  \; , $$
assuming $[0,1]$ is divided into intervals $[a_j,b_j)$ of length $h_j = b_j-a_j>0$ with $1 \leq j \leq k$
satisfying $b_j=a_{j+1}$ and $b_k=1, a_1=0$. 
Especially, if the spacing $h_j$ is the same everywhere, - in which case Riemann sums are 
often called Archimedean sums - then the exact formula 
$$   \sum_{k=1}^n f'(r_k) = n [f(b)-f(a)] \;  $$
holds. The mean value theorem establishes the existence of Rolle points:
$G(s) = f(x_0+a)-f(x_0-a) - 2 a f'(-a+2a s)$ satisfies $\int_0^1 G(s)\;ds=0$ so that
there is a solution $G(s)=0$ and so a root $x=x_0-a+2as$ of (\ref{impliciteequation}). We
take the point $s$ closest to $x_0$ and if there should be two, take the point to the right.
By definition, we have $g(0,x_0)=0$.  If the second derivative $f''(x_0)$ is nonzero, then the 
implicit function theorem assures that the map $a \to x(a)$ is a continuous function in $a$
for small $a$. The formula will indicate and confirm that near inflection points, the 
Rolle point $x(a)$ can move fast as a function of $a$. While it is also possible that Rolle point
can jump from one side to the other as we have defined it to be the closest, we can 
look at the motion of continuous path of $x_a$. \\

Assume $x_0=0$ and that we look at $f$ on the interval $[-a,a]$.  The formula 
$$  f(x_a) = f(0) + (\frac{f'(0) f'''(0) }{12 f''(0)^2}) f''(0) a^2 + O(a^4) $$
shown below in (\ref{rolleformula}) motivates to define the {\bf K-derivative} of a function $g$ as
$$ K_g(x) = \frac{g'(x) g'''(x)}{g''(x)^2} \; .  $$
It has similar features to the {\bf Schwarzian derivative}
$$  S_g(x)  = \frac{g'''(x)}{g'(x)} - \frac{3}{2} \frac{g''(x)}{g'(x))^2}  \;  $$
which is important in the theory of dynamical systems of one-dimensional maps. 
As the following examples show, they lead to similarly simple expressions, if we evaluate it for
some basic functions: \\

\begin{center} \begin{tabular}{|c|c|c|} \hline
  $g$                & Schwarzian $S_g$                   &  K-Derivative $K_g$                        \\ \hline
  $\cot(kx)$         & $2k^2$                             &  $1+\sec^2(kx)/2$                   \\
  $\sin(x)$          & $-1-3 \tan^2(x)/2$                 &  $-\cot^2(x)$                       \\
  $x^s$              & $(1-s^2)/(2x^2)$                   &  $(s-2)/(s-1)$                      \\
  $1/x$              & $0$                                &  $3/2$                              \\
  $\log(x)$          & $0$                                &  $2$                                \\
  $\exp(x)$          & $-1/2$                             &  $1$                                \\
  $\log(\sin(x))$    & $\frac{1}{2}\csc^2(x)[4-3 \sec ^2(x)]$ & $2\cos^2(x)$                    \\
  $\frac{ax+b}{cx+d}$& $0$                                &  $3/2$                              \\ \hline
\end{tabular} \end{center}

Both notions involve the first three derivatives of $g$. 
The $K$-derivative is invariant under linear transformations 
$$   K_{ag+b}(y) = K_g(y)   $$ 
and a constant for $g(x)=x^n, n \neq 0,1$ or $g(x)=\log(x)$ or $g(x)=\exp(x)$. 
Note that while the Schwarzian derivative is invariant under fractional linear
transformations
$$   S_{(ag+b)/(cg+d)}(y)  = S_g(y)  \; , $$
the K-derivative is only invariant under linear transformations only. It vanishes 
on quadratic functions. \\

An important example for the present zeta story is the function $g(x) = 1/\sin^s(x)$ for which
$$ K_g(x) = \frac{2 \cos^2(x) \left(s^2 \cos (2 x)+s^2+6 s+4\right)}{(s \cos (2 x)+s+2)^2} \;  $$
is bounded if $s>-1$. \\

\begin{comment}
CC[h_]:=Module[{},
  K[g_] := Simplify[g'[x] g'''[x]/(g''[x])^2];               k1[y_] := K[h] /. x -> y;
  S[g_] := Simplify[g'''[x]/g'[x] - (3/2) (g''[x]/g'[x])^2]; k2[y_] := S[h] /. x -> y;
  Simplify[{k2[x],k1[x]}]]; 

g[x_] := 1/(1+x^2);                    CC[g] 
g[x_] := 2 x+3 x^2;                    CC[g] 
g[x_] := Exp[x];                       CC[g] 
g[x_] := x^3;                          CC[g] 
g[x_] := (a Exp[x] + b)/(c Exp[x]+d);  CC[g] 
g[x_] :=Cot[k x];                      CC[g] 
g[x_] :=1/Sin[x]^(1/2);                CC[g] 
g[x_] :=1/Sin[x];                      CC[g] 
g[x_] :=1/Sin[x]^2;                    CC[g] 
g[x_] :=x^3-x;                         CC[g] 
g[x_] := (a x + b)/(c x + d);          CC[g] 
g[x_] :=Sin[x];                        CC[g] 
g[x_] :=1/x;                           CC[g] 
g[x_] :=1/x^2;                         CC[g] 
g[x_] :=Log[x];                        CC[g] 
g[x_] :=1/Sin[x]^s;                    CC[g] 
g[x_] :=1/Sqrt[Sin[Pi x]];             CC[g] 
g[x_] :=Log[x];                        CC[g] 
g[x_] :=Exp[x];                        CC[g] 
g[x_] :=Log[Sin[x]];                   CC[g] 
\end{comment}

We can now patch the classical Newton-Coates in situations, where the second derivative is large but the $K$-
derivative is small. This is especially useful for rational trigonometric functions for which the K-derivative 
is bounded, even so the functions are unbounded. 
The following theorem will explain the convergence of $(\zeta_n-{\rm E}[\zeta_n])/n^s$
which is scaled similarly than random variables in the central limit theorem for $s=1/2$. 
Here is the main result of this appendix: 

\begin{thm}[Rolle-Newton-Cotes type estimate]
Assume $g$ is $C^5$ on $I=(0,1)$ and has no inflection points $g''(x)=0 \in (0,1)$
and satisfies $|K_g(x)| \leq M$ for all $x \in (0,1)$.
Then, for every $0 \leq y <1$ and large enough $n$, 
\begin{eqnarray*}
 | \frac{1}{n} \sum_{k/n \in I} g(\frac{k+y}{n}) - \int_0^1 g(x) \; dx | 
     \leq \frac{1}{n^3} \sum_{k=1}^{n-1} M g''(\frac{k+y}{n}) \\
     \leq \frac{M_1}{n^2} [g'(\frac{n-(1/2)+y}{n}) - g'(\frac{y+1/2}{n})] \; . 
\end{eqnarray*}
\end{thm}
\begin{proof}
We can assume that $\int_0^1 g(x) \; dx = 0$ without loss of generality. Let $y/n+r_k$ be the Rolle points.
Then $\sum_{k/n \in I} g(r_k)=0$. Now use that
$|g(\frac{k+y}{n})|  \leq  M g''(\frac{k+y}{n})$, where $M$ is the upper bound for the K-derivative. 
The last equation is a Riemann sum. 
\end{proof}

This especially holds if $g$ is smooth in a neighborhood of the compact interval $[0,1]$ in which case the limit
\begin{equation}
 T(g)(y) = \lim_{n \to \infty} [ \sum_{k/n \in I} g(\frac{k+y}{n}) - n \int_a^b g(x) \; dx] 
 \label{renormalization1}
\end{equation}
exists and is linear in $y$. We will look at this ``central limit theorem" story elsewhere. 
The result has two consequences, which we will need:

\begin{coro}
For $g(x)=\sin(\pi x)^{-s} - (\pi x)^{-s}$ and $s>1$, we have
$$ \frac{1}{n^s} \sum_{\frac{k}{n} \in (0,1/2)}  g(\frac{k}{n}) \to 0  \; .  $$
\end{coro}
\begin{proof} 
The function $g$ is continuous on $(0,1/2)$. Therefore, the Riemann sum
$$  \frac{1}{n} \sum_{\frac{k}{n} \in (0,\frac{1}{2})}  g(\frac{k}{n}) $$
converges. 
\end{proof} 

The second one is:

\begin{coro}
For $g(x)=\sin(\pi x)^{-s}$ we know that
$$ [ \sum_{k/n \in I} g(\frac{k+y}{n}) - n \int_0^1 g(x) \; dx]/n^s  $$
converges uniformly for $s$ in compact subsets $K$ of the strip $0<{\rm Re}(s)<1$. 
\end{coro}
\begin{proof}
We have seen already that the K-derivative of $g$ is bounded. 
We have $G_n=[g'(\frac{n-1/2}{n}+\frac{y}{n}) - g'(\frac{1/2}{n}+\frac{y}{n})] \leq C_1$. 
Because the function $g''$ has always the same sign, the Rolle estimate shows that 
$|S_n/n-G_n| \leq |\sum_{k=1}^{n-1} g''(k/n) M/n^2|$. Applying the lemma again shows that this 
is the limit of $\frac{1}{n^{1+s}} [ g'(1-1/(2n)) - g'(1/(2n)) ]$. We see that we have 
a uniform majorant and so convergence to a function which is analytic inside the strip. 
\end{proof}

Given a smooth $L^1$ function $g$ on $(0,b)$ and $n \geq 2$, define 
$$  S_n(g)(y) = \sum_{0<k/n<1}  g(\frac{k+y}{n}) $$ 
and ${\rm E}[S_n]=\int_0^1 S_n(y) \; dy$ as well as 
${\rm Var}[S_n] = \int_0^1 (S_n-{\rm E}[S_n])^2 \; dx$ and
$\sigma[S_n]=\sqrt{{\rm Var}[S_n]}$. Define
$$ T_n(g) = \frac{S_n(g)-{\rm E}[S_n])}{\sigma(S_n)} \; . $$

\begin{figure}
\scalebox{0.21}{\includegraphics{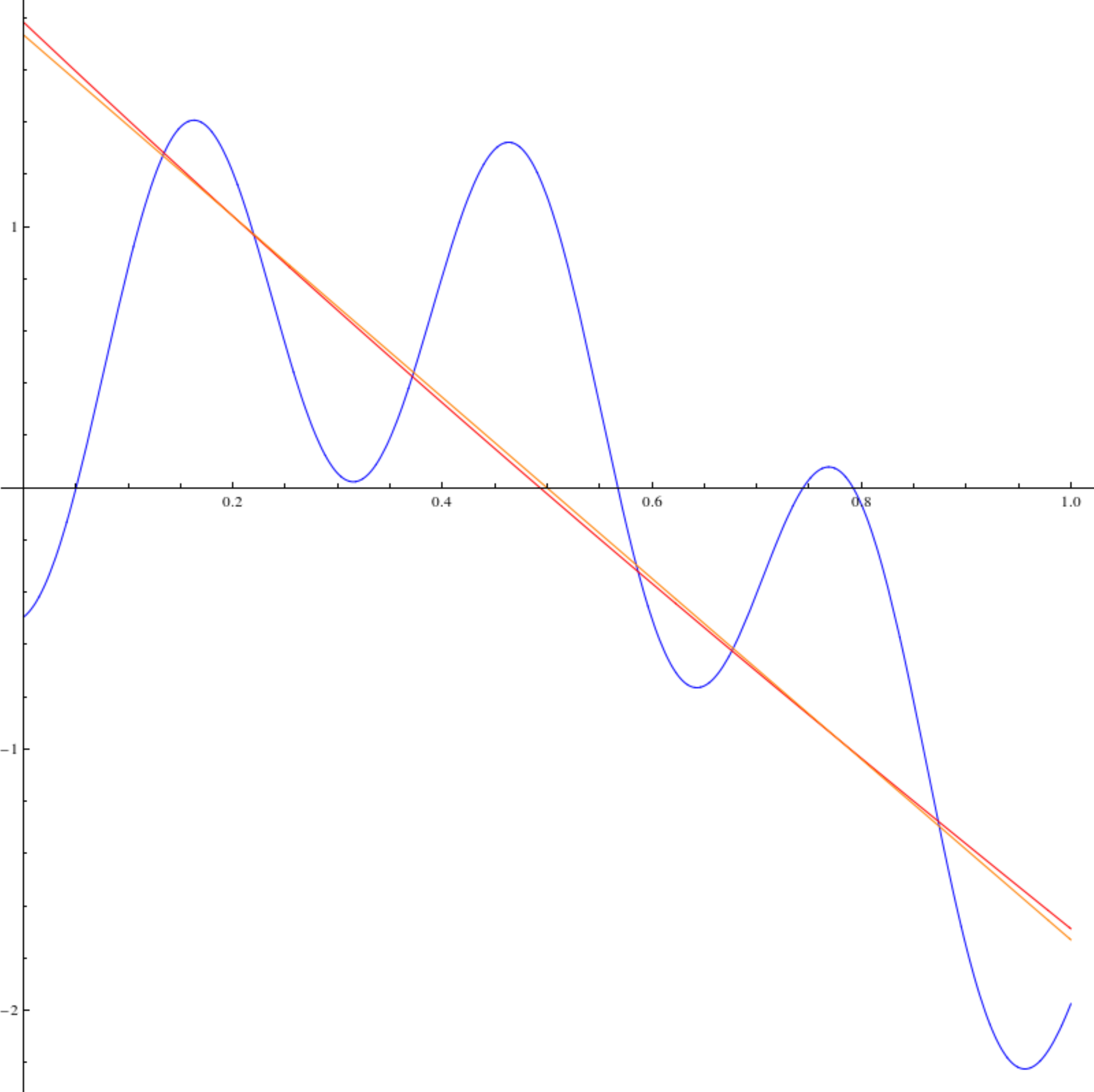}}
\scalebox{0.21}{\includegraphics{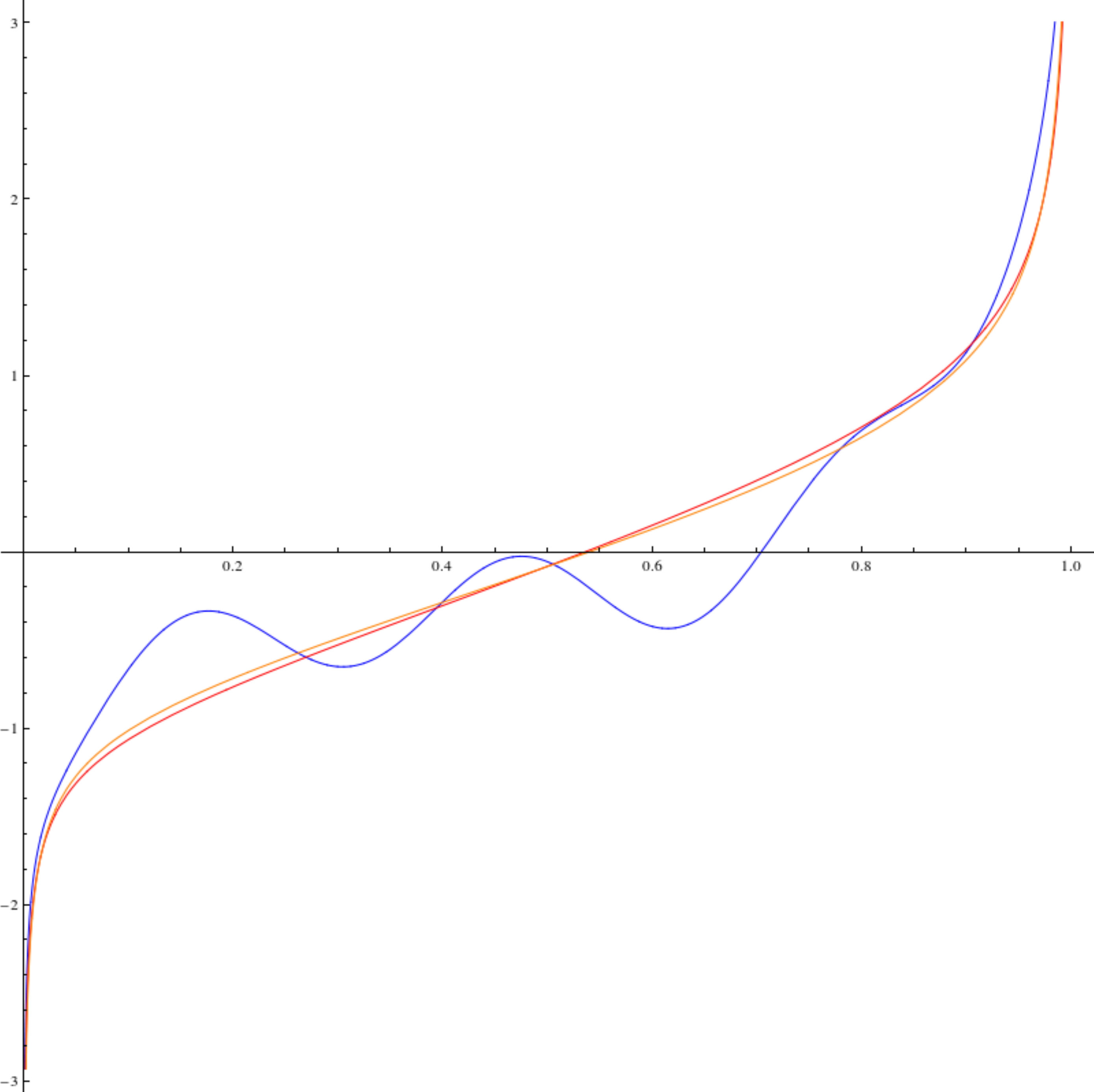}}
\caption{
Numerically we can iterate the operator $T$ on functions on $(0,1)$ twice.
For the picture, we used the renormalization step $T_{30}$ 
defined in (\ref{renormalization1}) twice. In the first picture, the initial 
function was a random smooth function of zero mean and variance $1$. 
The second example is a function for which $\lim_{x \to 0} g(x) (x-0)^{1/3}=1$ and 
$\lim_{x \to 1} g(x) (x-1)^{1/3} = -2$. The figures illustrate $T(Tg))=T(g)$ and that
for smooth functions in a neighborhood of $[0,1]$, the limit $T(g)$ is one of
the two linear functions $ax+b = \pm \sqrt{12} x-\sqrt{3}$. If the function
is unbounded at the ends, then the limiting function can become nonlinear.  }
\label{centrallimittheorem}
\end{figure}

{\bf Examples.} \\
{\bf 1)} The function $g(x)=1/x$ on $(1/n,1)$ has the K-derivative $3/2$. The Riemann sum
is $(1/n) \sum_{k=1}^{n} n/k = \sum_{k=1}^{n} 1/k$. The integral is $\log(n)$. The theorem tells
that $\sum_{k=1}^n 1/k - \log(n)$ is bounded. Of course, we know it converges to $\gamma$ 
the Euler-Mascheroni constant. \\
{\bf 2)} Lets take $g(x) = \log(x)$ on  $I=[1,b)$ for which the K-derivative $K_g(x) = 3/2$ 
independent of $\epsilon$.  Assume $y=0$. 
The classical Newton-Coates would not have give us a good estimate for the Riemann sum because the interval
increases. However, the above theorem works. 
Applying the above theorem for $n=1$ gives for every $b>1$
$$ \sum_{k<b} \frac{1}{k} - \log(b)   \leq  2+1 =3 \; . $$
Indeed the limit for $b \to \infty$ exists and is again {\bf Euler-Mascheroni constant}. But the Rolle estimate
is uniform in $n$. The bound $3$ on the right hand side is not optimal. One can show it to be $\leq 1$.  \\
{\bf 3)} For $g(x) = \cot(\pi x)$, we have ${\rm min}( \{|K_f(x) f''(x)|/6,|f''(x)/12|\}) \leq 1$. Again assume first $y=0$. 
We have $|\sum_{k=1}^{[(n-1)/2]} \cot(\pi k/n)/n + \log( \sin(\pi x)/\pi)| \leq 1$ for all $n$. 
The limit which exists  according to the theorem is about $0.183733$.  \\
{\bf 4)} The theorem implies that the limit $(1/n) \sum_{k=1}^{n-1} \cot(\pi (k+y)/n)$ exists. Actually, 
since this is the Birkhoff normalization operator (a special Ruelle transfer operator known in dynamical systems
theory), the limit is explicitly known to be $\cot(\pi y)$ if we sum from $k=0$ to $n-1$. \\
{\bf 5)} Let $g(x) = \sin^{-s}(\pi x)$ with $0<s<1$, where $M$ is $1$. We also have 
$$ \lim_{n \to \infty} \frac{g(\frac{1}{n})}{n^s} = \frac{1}{\pi^s} $$ 
converging. The theorem applies and the limit exists because for $g(x) = \sin(\pi x)^{-s}$, 
the Rolle function is bounded. 

\begin{comment}
g[x_] := 1/Sin[Pi x]^s;
Sin[Pi x]^(2 + s) g'[x] g'''[x]/g''[x] // Simplify
\end{comment}

\begin{lemma}[Rolle point lemma]
Assume $f$ is 5 times continuously differentiable in an interval
$I = (-b,b)$ and that $f''(0)$ is not zero. Then
there is $\xi \in I$ such that the Rolle point $x(a)$ satisfies
$$ x(a) = a^2 \frac{f'''(0)}{6 f''(0)} + a^4 \frac{f'''''(\xi)}{120 f''(0)} \; .  $$
Furthermore,
$$ f(x_a) = f(0) + a^2 K_f(0) \frac{f''(0)}{6} + O(a^4)   \; . $$
\label{rolleformula}
\end{lemma} 

\begin{proof}
We start with the implicit equation 
$$ g(x,a) = f(a) - f(-a) - 2 a f'(x) = 0 \;  $$
from which the chain rule gives in a familiar way the implicit 
differentiation formula $x'(a) = -g_a(0,0)/g_x(0,0)$. 
A Taylor expansion with Lagrangian rest term
$$  f(a) = f(0) + a f'(0) + \frac{a^2 f''(0)}{2!} 
         + \frac{a^3 f'''(0)}{3!} + \frac{a^4 f''''(0)}{4!} + \frac{a^5 f'''''(\xi)}{5!}  \; $$
for some $\xi \in (-b,b)$ shows that
\begin{eqnarray*}
  g_a(0,0) &=& 2 f'(0)  +  2 \frac{a^3 f'''(0)}{3!}  + 2 \frac{a^5 f'''''(\xi)}{5!}  - 2 f'(0) \\
  g_x(0,0) &=& -2a f''(0)   \; . 
\end{eqnarray*}
Therefore, $x'(a) = a^2 f'''(0)/(6 f''(0)) + a^4 f'''''(\xi)/(120 f''(0))$. And the rest is clear using 
$f(x_a) = f(0) + f'(0) x_a + f''(0) x_a^2/2 + \dots $. 
\end{proof}

{\bf Remark.} \\
Using Riemann sums using points which estimate the Rolle points can produce a numerical method for integration of
functions which are 5 times continuously differentiable
in the interior of some interval but possibly unbounded at the boundary. Let $x_k$ denote
the end points of a Riemann partition and $y_k$ the midpoints and $2a_k$ the length of the interval. Then 
$$  S_nf = \sum_{k=1}^n [f(y_k) a_k  + \frac{K_f(y_k)}{6} f''(y_k) a_k^3]   \;  $$
has adjusted the midpoint in every interval $[-a,a]$ by an approximation of $f(x_a)$. Of course, there
is a trade off, since we have to compute the $K$-derivative $K_f(y_k)$ at a point. But the result will be of the
order $a^5$ close to the real integral, like Simpson. 
Why is this helpful? It is the fact that the $K$-derivative is often finite, even so the function and its 
derivative can be unbounded. 
Any method, where intervals are divided in a definite way, like Simpson, can have larger error in those
cases. An example is the function $f(x) = \sin(\pi x)^{-s}$ for $s \in (0,1)$ which is integrable but where 
standard integration methods complain about the singularity. Actually, one of our results shows that the difference of 
Riemann sum and integral is of the order $n^2$. The function, which appeared in the discrete circle zeta
function, has a finite $K$-derivative. 

\vspace{12pt}
\bibliographystyle{plain}

\begin{thebibliography}{10}

\bibitem{fanchunglattice}
F.~Chung.
\newblock Spanning trees in subgraphs of lattices.
\newblock 2000.

\bibitem{Conway1978}
J.B. Conway.
\newblock {\em Functions of One Complex Variable}.
\newblock Springer Verlag, 2. edition, 1978.

\bibitem{Cooper}
Y.~Cooper.
\newblock Properties determined by the ihara zeta function of a graph.
\newblock {\em Electronic Journal of Combinatorics}, 16, 2009.

\bibitem{Dowker}
J.S. Dowker.
\newblock Heat-kernel on the discrete circle and interval.
\newblock http://www.arxiv.org/abs/1207.20966, 2012.

\bibitem{Elizalde}
E.~Elizalde.
\newblock {\em Ten Physical Applications of Spectral Zeta Functions}.
\newblock Lecture Notes in Physics. Springer, 1995.

\bibitem{GonekLedoan}
S.M. Gonek and A.H.Ledoan.
\newblock Zeros of partial sums of the {R}iemann zeta-function.
\newblock {\em Int. Math. Res. Not. IMRN}, (10):1775--1791, 2010.

\bibitem{HL}
G.~H. Hardy and J.~E. Littlewood.
\newblock Some problems of {D}iophantine approximation: a series of cosecants.
\newblock {\em Bulletin of the Calcutta Mathematica Society}, 20(3):251--266,
  1930.

\bibitem{HL46}
G.~H. Hardy and J.~E. Littlewood.
\newblock Notes on the theory of series. {XXIV}. {A} curious power-series.
\newblock {\em Proc. Cambridge Philos. Soc.}, 42:85--90, 1946.

\bibitem{Ihara}
Y.~Ihara.
\newblock On discrete subgroups of the two by two projective linear graop over
  p-adic fields.
\newblock {\em J. Math. Soc. Japan}, 18:219--235, 1966.

\bibitem{JohnsonTucker}
T.~Johnson and W.~Tucker.
\newblock Enclosing all zeros of an analytic function - a rigorous approach.
\newblock {\em Journal of Computational and Applied Mathematics}, 228:418--423,
  2009.

\bibitem{kirsten}
K.~Kirsten.
\newblock Basic zeta functions and some applications in physics.
\newblock {\em A window into zeta and modular physics}, 57, 2010.

\bibitem{knillcotangent}
O.~Knill.
\newblock Selfsimilarity in the birkhoff sum of the cotangent function.
\newblock http://arxiv.org/abs/1206.5458, 2012.

\bibitem{brouwergraph}
O.~Knill.
\newblock A {Brouwer} fixed point theorem for graph endomorphisms.
\newblock {\em Fixed Point Theory and Applications}, 85, 2013.

\bibitem{cauchybinet}
O.~Knill.
\newblock Cauchy-{B}inet for pseudo determinants.
\newblock {\\}http://arxiv.org/abs/1306.0062, 2013.

\bibitem{KL}
O.~Knill and J.~Lesieutre.
\newblock Analytic continuation of {D}irichlet series with almost periodic
  coefficients.
\newblock {\em Complex Analysis and Operator Theory}, 6(1):237--255, 2012.

\bibitem{KT}
O.~Knill and F.~Tangerman.
\newblock Selfsimilarity and growth in {Birkhoff} sums for the golden rotation.
\newblock {\em Nonlinearity}, 21, 2011.

\bibitem{Mantuano}
T.~Mantuano.
\newblock Discretization of {R}iemannian manifolds applied to the {H}odge
  {L}aplacian.
\newblock {\em Amer. J. Math.}, 130(6):1477--1508, 2008.

\bibitem{Mora}
G.~Mora.
\newblock On the asymptotically uniform distribution of the zeros of the
  partial sums of the {R}iemann zeta function.
\newblock {\em J. Math. Anal. Appl.}, 403(1):120--128, 2013.

\bibitem{ruellezeta}
D.~Ruelle.
\newblock {\em Dynamical Zeta Functions for Piecewise Monotone Maps of the
  Interval}.
\newblock CRM Monograph Series. AMS, 1991.

\bibitem{Terras}
A.~Terras.
\newblock {\em Zeta functions of Graphs}, volume 128 of {\em Cambridge studies
  in advanced mathematics}.
\newblock Cambridge University Press.

\end{thebibliography}

\end{document}